\documentclass[11pt]{amsart}

\usepackage{a4wide,amsfonts,amssymb,stmaryrd,amscd,amsmath,latexsym,amsbsy,hyperref}

\usepackage[all]{xy}

\theoremstyle{plain}

\newcommand{\CC}{\mathbb C}

\newcommand{\ZZ}{\mathbb Z}

\newcommand{\solu}[1]{\begin{sol}{\bf (\ref{#1})}}

\numberwithin{equation}{section}

\theoremstyle{plain}
\newtheorem{theorem}[subsubsection]{Theorem}
\newtheorem{lemma}[subsubsection]{Lemma}
\newtheorem{prop}[subsubsection]{Proposition}
\newtheorem{cor}[subsubsection]{Corollary}
\newtheorem{conj}[subsubsection]{Conjecture}
\newtheorem{claim}[subsubsection]{Claim}

\theoremstyle{definition}
\newtheorem{defn}[subsubsection]{Definition}

\newtheorem{remark}[subsubsection]{Remark}
\newtheorem{exam}[subsubsection]{Example}

\def\AA{\mathbb{A}}

\def\CC{\mathbb{C}}

\def\GG{\mathbb{G}}

\def\NN{\mathbb{N}}

\def\PP{\mathbb{P}}
\def\QQ{\mathbb{Q}}

\def\ZZ{\mathbb{Z}}

\def\calA{\mathcal{A}}

\def\calE{\mathcal{E}}
\def\calF{\mathcal{F}}

\def\calH{\mathcal{H}}

\def\calJ{\mathcal{J}}

\def\calL{\mathcal{L}}

\def\calO{\mathcal{O}}

\def\bR{\mathbf{R}}

\renewcommand\a\alpha
\renewcommand\b\beta
\newcommand{\g}{\gamma}
\renewcommand\d\delta
\newcommand\D\Delta
\renewcommand{\th}{\theta}
\newcommand{\ph}{\varphi}
\newcommand{\s}{\sigma}
\renewcommand{\t}{\tau}
\newcommand{\x}{\xi}
\newcommand{\y}{\eta}

\newcommand{\ep}{\epsilon}
\renewcommand{\l}{\lambda}
\renewcommand{\L}{\Lambda}
\newcommand{\om}{\omega}

\newcommand{\fra}{\mathfrak{a}}
\newcommand{\frh}{\mathfrak{h}}
\newcommand{\frg}{\mathfrak{g}}

\newcommand{\fm}{\mathfrak{m}}

\newcommand{\frH}{\mathfrak{H}}
\newcommand{\frL}{\mathfrak{L}}

\newcommand{\can}{\textup{can}}
\newcommand{\ch}{\textup{ch}}
\newcommand{\codim}{\textup{codim}}

\newcommand\et{\text{\'et}}

\newcommand{\Gr}{\textup{Gr}}

\newcommand\id{\textup{id}}
\renewcommand{\Im}{\textup{Im}}

\newcommand\Jac{\textup{Jac}}

\newcommand{\red}{\textup{red}}

\newcommand\rk{\textup{rk}}

\newcommand\Spec{\textup{Spec}}

\newcommand\Sym{\textup{Sym}}
\newcommand{\Todd}{\textup{Todd}}
\newcommand{\tor}{\textup{tor}}
\newcommand{\Tr}{\textup{Tr}}

\newcommand\triv{\textup{triv}}

\newcommand{\univ}{\textup{univ}}

\newcommand\Aut{\textup{Aut}}
\newcommand\Hom{\textup{Hom}}
\newcommand\End{\textup{End}}

\newcommand\GL{\textup{GL}}

\newcommand\PGL{\textup{PGL}}

\newcommand\SL{\textup{SL}}
\def\sl{\mathfrak{sl}}

\newcommand\Sp{\textup{Sp}}

\newcommand{\Ad}{\textup{Ad}}

\newcommand{\isom}{\stackrel{\sim}{\to}}
\newcommand\incl{\hookrightarrow}
\newcommand\surj{\twoheadrightarrow}
\newcommand\bij{\leftrightarrow}
\newcommand{\leftexp}[2]{{\phantom{#2}}^{#1}{#2}}

\newcommand{\jiao}[1]{\langle{#1}\rangle}

\newcommand{\pt}{\textup{pt}}

\newcommand{\wt}{\widetilde}

\newcommand\sss{\subsubsection}
\newcommand\xr{\xrightarrow}

\newcommand{\Cha}{Cherednik algebra }

\newcommand{\Hrat}{\mathfrak{H}}

\newcommand{\HrtE}{\mathfrak{H}_{q/p,\epsilon=1}}

\newcommand{\ptau}{{}^{p}{\tau}}
\newcommand{\upH}{\textup{H}}
\newcommand{\cohog}[2]{\upH^{#1}({#2})}     

\newcommand{\eqcoh}[2]{\upH^{#1}_{\Gm}({#2})}

\newcommand{\spcoh}[2]{\upH^{#1}_{\ep=1}({#2})}  

\newcommand{\Fl}{\textup{Fl}}

\newcommand{\Hilb}{\textup{Hilb}}

\newcommand{\cPic}{\overline{\textup{Pic}}}

\newcommand{\cJ}{\overline{\Jac}}
\newcommand{\SpGr}{\Sp^{\Gr}}

\newcommand{\Gm}{\GG_m}

\newcommand{\rot}{\textup{rot}}
\newcommand{\dil}{\textup{dil}}
\newcommand{\Grot}{\Gm^{\rot}}
\newcommand{\Gdil}{\Gm^{\dil}}

\newcommand{\Eun}{\calE^{\univ}}
\newcommand{\Hun}{\calH^{\univ}}

\newcommand{\barM}{\overline{M}}

\newcommand\pb{\textup{par}}
\newcommand\pure{\textup{pur}}

\newcommand\cA{\mathcal{A}}
\newcommand\cB{\mathcal{B}}
\newcommand\cC{\mathcal{C}}

\newcommand\cE{\mathcal{E}}
\newcommand\cF{\mathcal{F}}

\newcommand\cH{\mathcal{H}}

\newcommand\cK{\mathcal{K}}
\newcommand\cL{\mathcal{L}}
\newcommand\cM{\mathcal{M}}

\newcommand\cO{\mathcal{O}}
\newcommand\cP{\mathcal{P}}
\newcommand\cQ{\mathcal{Q}}

\newcommand\cY{\mathcal{Y}}

\newcommand\un{\underline}

\newcommand{\ov}{\overline}

\newcommand{\tl}[1]{[\![#1]\!]}
\newcommand{\lr}[1]{(\!(#1)\!)}

\newcommand\SG{\SpGr_{q/p}}

\begin{document}

\title{The cohomology ring of certain compactified Jacobians}

\pagestyle{plain} 

\author{Alexei Oblomkov}
\email{oblomkov@math.umass.edu}
\address{Department of Mathematics, University of Massachusetts at Amherst, LGRT, 710 N. Pleasant St.,
  Amherst, MA 01003}
\author{Zhiwei Yun}
\email{zhiwei.yun@yale.edu}
\address{Department of Mathematics, Yale University, 10 Hillhouse Ave, New Haven, CT 06511}
\date{}
\subjclass[2010]{Primary 14F20, 14F40; Secondary 14D20}
\keywords{Cherednik algebras; Hitchin fibration; affine Springer fibers}

\begin{abstract} We provide an explicit presentation of the equivariant cohomology ring of the compactified Jacobian $J_{q/p}$ of the rational curve $C_{q/p}$ with planar equation $x^{q}=y^{p}$ for $(p,q)=1$. We also prove analogous results for the closely related affine Springer fiber $\Sp_{q/p}$ in the affine flag variety of $\SL_{p}$. We show that the perverse filtration on the cohomology of $J_{q/p}$ is multiplicative, and the associated graded ring under the perverse filtration is a degeneration of the ring of functions on a moduli space of maps $\PP^{1}\to C_{q/p}$. We also propose several conjectures about $J_{q/p}$ and more general compactified Jacobians. 
\end{abstract}

\maketitle

\tableofcontents

\section{Introduction}

\subsection{Main results} Let $p$ and $q$ be coprime positive integers. Let $C=C_{q/p}$ be the compactification of the plane curve $y^{p}-x^{q}$ over $\CC$ which is smooth at $\infty$. When $p,q>1$, the only singularity of $C_{q/p}$ is the point $(x,y)=(0,0)$. Let $J_{q/p}$ be the compactified Jacobian of $C_{q/p}$. It classifies rank one torsion-free coherent sheaves on $C_{q/p}$ with a fixed degree.  In this paper we study the cohomology ring of $J_{q/p}$. The notation $\cohog{*}{Y}$ will denote the singular cohomology of a complex variety with $\QQ$-coefficients.

\subsubsection{Presentation of the cohomology ring}\label{sss:Ex}
Let $\pi_{x}: C_{q/p}\to\PP^{1}$ be the projection $(x,y)\mapsto x$. Taking direct image of a torsion-free rank one sheaf on $C_{q/p}$ under $\pi_{x}$ gives a rank $p$ vector bundle on $\PP^{1}$. Let $\cF^{\univ}$ be the universal sheaf on $J_{q/p}\times C_{q/p}$, and let $\Eun_{x}$ be its direct image under $\id_{J_{q/p}}\times \pi_{x}: J_{q/p}\times C_{q/p}\to J_{q/p}\times\PP^{1}$. Then $\Eun_{x}$ is a vector bundle of rank $p$ on $J_{q/p}\times \PP^{1}$. We write the Chern classes of $\Eun_{x}$ in terms of K\"unneth decomposition $c_{i}(\Eun_{x})=d_{i}\otimes1+e_{i}\otimes[\PP^{1}]$, where $d_{i} \in\cohog{2i}{J_{q/p}}$ (in fact $d_{i}=0$) and  $e_{i}\in \cohog{2i-2}{J_{q/p}}$. In Section~\ref{s:gen} we show:

\begin{theorem}\label{th:gen} The cohomology ring $\cohog{*}{J_{q/p}}$ is generated by $e_{2},\cdots, e_{p}$ as a $\QQ$-algebra. 
\end{theorem}

Note that $C_{q/p}$ admits a $\Gm$-action given by $\Gm\ni t: (x,y)\mapsto (t^{p}x,t^{q}y)$, which induces an action on $J_{q/p}$. In this paper we denote this $\Gm$ by $\Gm(q/p)$. We may then consider the equivariant cohomology $\eqcoh{*}{J_{q/p}}$.  The localized equivariant cohomology of $J_{q/p}$ (namely when the equivariant parameter $\ep\in\eqcoh{2}{\pt}$ is inverted) is easy to compute by passing to the fixed points, and the corresponding generation statement readily follows. The generation statement in the non-equivariant case is much more subtle: the proof uses the geometry of moduli spaces of Higgs bundles and the corresponding generation statement there following the work of Hausel-Thaddeus \cite{HT} and Markman \cite{Mark}.

We give a presentation of the equivariant cohomology ring of $J_{q/p}$.

\begin{theorem}[See Proposition~\ref{p:eqcoh}]\label{th:HSp}  We have a graded ring isomorphism
\begin{equation*}
\QQ[\ep, e_{2},\cdots, e_{p},f_{2},\cdots, f_{q}]/I^{sat}_{q/p}\cong   \eqcoh{*}{J_{q/p}}
\end{equation*}
where $\ep$ is the equivariant parameter, and $I^{sat}_{q/p}$ is the $\ep$-saturation of the ideal generated by the coefficients (for the variable $z$) in the expression
$$\prod_{j=0}^{q-1}A(z+jp\ep)-\prod_{i=0}^{p-1}B(z+iq\ep),$$
where 
$$A(z)=z^{p}+\frac{1}{2}pq(p-1)z^{p-1}+\sum_{i=2}^{p}p\ep e_{i}z^{p-i}; \quad B(z)=z^{q}+\frac{1}{2}pq(q-1)z^{q-1}+\sum_{j=2}^{q}q\ep f_{j}z^{q-j}.$$ 
We declare that $e_{i}$ has degree $2i-2$ ($2\leq i\leq p$) and $f_{i}$ has degree $2j-2$ ($2\leq j\leq q$).
\end{theorem}

\subsubsection{The perverse filtration on $\cohog{*}{J_{q/p}}$}  For any locally planar curve $C$, there is an increasing filtration $P_{\le i}$ on the cohomology $\cohog{*}{\cJ(C)}$ of the compactified Jacobian of $C$. This filtration was introduced by Maulik and the second-named author in \cite{MY}, and we call it the {\em perverse filtration}. The construction of the perverse filtration will be recalled in Section~\ref{sss:perv}. It uses the extra choice of a deformation of the curve satisfying certain conditions, but the resulting filtration is independent of such deformations.  For smooth curves, the perverse filtration is simply induced by the cohomological grading. 

\begin{theorem}\label{th:perv mult} The perverse filtration on $\cohog{*}{J_{q/p}}$ is multiplicative under the cup product. In other words, for $i,j\in \ZZ$ we have
\begin{equation*}
P_{\le i}\cohog{*}{J_{q/p}}\cup P_{\le j}\cohog{*}{J_{q/p}}\subset P_{\le i+j}\cohog{*}{J_{q/p}}.
\end{equation*}
\end{theorem}

For the proof, we introduce a ``Chern filtration'' $C_{\le *}\cohog{*}{J_{q/p}}$ using polynomials of the equivariant Chern classes of $\cE^{\univ}_{x}$, which is obviously by definition. We then show that the perverse filtration $P_{\leq *}$ coincides with the Chern filtration $C_{\le *}$ using results from \cite{OY}, which ultimately relies on the rational Cherednik algebra action on the equivariant cohomology of affine Springer fibers. We will discuss this connection in Section~\ref{intro ASF}.


\sss{Comparison with $\cO_{q/p}$}  Consider the moduli space $M^{rig}_1(\PP^1,C_{q/p})$ over $\QQ$ classifying rigidified maps from $\PP^{1}$ to $C_{q/p}$. This is a fat point whose coordinate ring we denote by $\cO_{q/p}$ (a local artinian ring with residue field $\QQ$) . The action of $\Gm$ on $C_{q/p}$ induces an action on $M^{rig}_1(\PP^1,C_{q/p})$, hence a grading  $\cO_{q/p}=\oplus_{j\in \ZZ}\cO_{q/p}[j]$. We propose the following conjecture relating $\cohog{*}{J_{q/p}}$ and $\cO_{q/p}$.

\begin{conj}\label{conj:Gr m} Let $\fm\subset \cO_{q/p}$ be the maximal ideal. Then for any $i,j\in\ZZ$ there is a canonical isomorphism
\begin{equation*}
\Gr^{P}_{j}\cohog{2i}{J_{q/p}}\cong \Gr^{j-i}_{\fm}\cO_{q/p}[j].
\end{equation*}
Moreover, these isomorphisms are compatible with the ring structures on both sides (the ring structure on the left side is guaranteed by the multiplicativity of the perverse filtration as shown in Theorem~\ref{th:perv mult}).
\end{conj}
This conjecture says that the rings $\cohog{*}{J_{q/p}}$ and $\cO_{q/p}$ have a common degeneration in a precise way.  We have the following partial results towards the above conjecture.

\begin{theorem}\label{th:main}
\begin{enumerate}
\item There is an increasing filtration $\{F_{\le i}\cO_{q/p}\}$ on $\cO_{q/p}$, compatible with the grading and multiplicative with respect to the ring structure on $\cO_{q/p}$, such that there are canonical isomorphisms for $i,j\in\ZZ_{\ge0}$ 
\begin{equation*}
\Gr^{P}_{j}\cohog{2i}{J_{q/p}}\cong \Gr^{F}_{i}\cO_{q/p}[j].
\end{equation*}
Moreover, these isomorphisms are compatible with the ring structures on both sides.
\item We have
\begin{equation}\label{F contains m}
F_{\le i}\cO_{q/p}[j]\supset \fm^{j-i}\cap \cO_{q/p}[j]
\end{equation}
for all $i,j\in\ZZ_{\ge0}$ (when $j<i$, we understand $\fm^{j-i}=\cO_{q/p}$). The equality holds if and only if for all $i\ge0$
\begin{equation}\label{H2i Grm}
\dim\cohog{2i}{J_{q/p}}=\dim\Gr^{\d-i}_{\fm}\cO_{q/p}, \quad \textup{where }\d=(p-1)(q-1)/2.
\end{equation} 
\end{enumerate}
\end{theorem}

From the above theorem we immediately get
\begin{cor}\label{c:equiv conj} Conjecture~\eqref{conj:Gr m} holds if and only if \eqref{H2i Grm} holds for all $i\ge 0$.
\end{cor}

The proof Theorem~\ref{th:main} is by observing the similarities between a presentation of  $\cO_{q/p}$ (given in \cite{FGS}) and the presentation of $\cohog{*}{J_{q/p}}$ given in Theorem~\ref{th:HSp} above. In fact, there is an algebra $R$ over the two variable polynomial ring $\QQ[\ep,s]$ whose various specializations recover $\spcoh{}{J_{q/p}}$, $\cohog{*}{J_{q/p}}$ and $\cO_{q/p}$. We also make the following observation.


\begin{theorem}\label{th:Gr m} Conjecture~\ref{conj:Gr m} holds if the ring $R$ introduced in Section~\ref{sss:R} is flat over $\QQ[\ep,s]$.
\end{theorem}
In the special case $q=kp+1$, we relate Conjecture \ref{conj:Gr m} to rational Cherednik algebras and the punctual Hilbert scheme. We formulate Conjecture~\ref{conj:H0Hilb} which gives a character formula for a finite dimensional representation of the spherical rational Cherednik algebra in terms of the punctual Hilbert scheme;  this character formula implies Conjecture~\ref{conj:Gr m}. 

\subsection{Relation with affine Springer fibers}\label{intro ASF} 
The compactified Jacobian $J_{q/p}$ has several other incarnations that are useful in our proof and in applications.  These incarnations can be summarized into the following diagram
\begin{equation*}
\xymatrix{  \Gr_{\SL_{p}} &  \Sp^{\Gr}_{q/p}\ar@{_{(}->}[l]\ar[r]^{\sim} &  X_{q/p} \ar[r] &  J_{q/p}\ar@{^{(}->}[r]&   M_{\SL_{p}}}
\end{equation*}
Here $X_{q/p}$ is the moduli space of $\CC\tl{t^{p},t^{q}}$-lattices inside $\CC\lr{t}$ with a fixed volume. There is a natural map $X_{q/p}\to J_{q/p}$ which is a homeomorphism. We can realize $X_{q/p}$ as an affine Springer fiber $\Sp^{\Gr}_{q/p}$ inside the affine Grassmannian $\Gr_{\SL_{p}}$ for the group $\SL_{p}$. At the right end of the diagram, we realize $J_{q/p}$ as a fiber of the Hitchin fibration for the moduli space $M_{\SL_{p}}$ of stable (twisted) Higgs bundles of rank $p$ over a weighted projective curve. We remark that the roles of $p$ and $q$ are symmetric: we may as well fit $\Gr_{\SL_{q}}$ and $\cM_{\SL_{q}}$ in the above diagram.

The incarnation of $J_{q/p}$ as a Hitchin fiber allows us to use global geometric argument in \cite{HT} and \cite{Mark} to prove the generation statement (Theorem~\ref{th:gen}). The incarnation of $J_{q/p}$ as an affine Springer fiber allows us to use the results of \cite{VV} and \cite{OY} to see the double affine Hecke symmetry on its cohomology, which plays a key role in the proof of the multiplicativity of the perverse filtration (Theorem~\ref{th:perv mult}).

The affine Springer fiber $\SG$ has a close relative $\Sp_{q/p}$ in the affine flag variety (see Section~\ref{sss:aff flag}). There is a natural projection $\Sp_{q/p}\to \SG\to J_{q/p}$. Also there are tautological line bundles $\cL_{1},\cdots, \cL_{p}$ on $\Sp_{q/p}$ which are restrictions of line bundles on the affine flag variety (see Section~\ref{sss:aff flag}). We prove an analogue of Theorem~\ref{th:gen} for $\Sp_{q/p}$.

\begin{theorem}\label{th:par gen} The cohomology ring $\cohog{*}{\Sp_{q/p}}$  is generated by the pullbacks of $e_{2},\cdots, e_{p}$ and the Chern classes $c_{1}(\cL_{1}),\cdots, c_{p}(\cL_{p})$ over $\QQ$.
\end{theorem}

In Section~\ref{ss:coho aff flag}, we propose a conjectural description of the relations defining $\cohog{*}{\Sp_{q/p}}$ (Conjecture~\ref{conj:rel Sp Fl}) which is an analogue of Theorem~\ref{th:HSp}.  We prove a weak version of this conjecture, see Proposition~\ref{p:par rel}. 

\subsection{Other conjectures} We make several other speculations that are scattered in the main body of the paper.

In Section~\ref{ssec:WeightFilt}, we construct a smooth variety that homotopy retracts to $J_{q/p}$. We then conjecture that this smooth model of $J_{q/p}$ should be Poisson and we give a conjectural description of its symplectic leaves. We speculate that this smooth model is related to a certain wild character variety.

In Section~\ref{sec:conjectures} we formulate weak versions of Conjecture~\ref{conj:Gr m} for
toric curve singularities, see Conjecture~\ref{conj:toric}, and for planar curve singularities, see Conjecture~\ref{conj:planar}.

{\bf Acknowledgements}  The authors would like to thank Roman Bezrukavnikov, Eugene Gorsky, Davesh Maulik, Eyal Markman and Vivek Shende for useful discussions related to the paper.
A.O. was partially supported by Sloan Foundation and NSF grant NSF CAREER grant DMS 1352398; Z.Y. was supported by the NSF grant DMS-1302071/DMS-1736600 and the Packard Fellowship.

\section{Geometry of $C_{q/p}$ and $J_{q/p}$}
\subsection{The curve $C_{q/p}$}\label{ss:Cpq} Let $p$ and $q$ be a pair of coprime positive integers. Let $\PP(1,p,q)$ be the weighted projective plane obtained as the quotient of $\AA^{3}-\{0\}$ by the action of $\Gm$ given by $t\cdot (z_{0},z_{1},z_{2})=(tz_{0},t^{p}z_{1}, t^{q}z_{2})$, where $t\in \Gm$ and $(z_{0},z_{1},z_{2})\in\AA^{3}$. The equation $z^{q}_{1}-z^{p}_{2}$ defines a curve $C_{q/p}\subset \PP(1,p,q)$. Let $\infty_{C}=[0:1:1]\in C_{q/p}$. Then the complement $C_{q/p}-\{\infty_{C}\}$ is the affine plane curve with affine coordinates $x=z_1/z_0^q,y=z_2/z_0^p$  and given by the equation $x^{q}-y^{p}$. The only singularity of $C_{q/p}$ is at $(x,y)=(0,0)$, i.e., the point $[1:0:0]\in \PP(1,p,q)$.

We fix the normalization map $\phi_{0}:\PP^{1}\to C_{q/p}$ given in affine coordinates by $t\mapsto (t^{q},t^{p})$. Then $\phi(\infty)=\infty_{C}$. There is an action of $\Gm$ on $C_{q/p}$ given by $t\cdot(x,y)=(t^{p}x,t^{q}y)$. We denote this one-dimensional torus by $\Gm(q/p)$.

\subsection{The moduli space of maps from $\PP^{1}$}
Consider the moduli functor $M_1(\PP^1,C_{q/p})$ on $\QQ$-algebras classifying degree $1$ maps from $\PP^1$ to $C_{q/p}$.  The automorphism group of $\PP^1$ acts on $M_1(\PP^1,C_{q/p})$, and the reduced structure of $M_1(\PP^1,C_{q/p})$ is a torsor under $\PGL_{2}=\Aut(\PP^{1})$ because any degree one map $\PP^{1}\to C_{q/p}$ must be the normalization map. 

We define a slice of this action.  Let $M_1((\PP^1,\infty),(C_{q/p},\infty_{C}))$ be the moduli space of maps $\PP^{1}\to C_{q/p}$ that maps $\infty$ to $\infty_{C}$. For any $\QQ$-algebra $A$, an $A$-point $\phi$ of $M_1((\PP^1,\infty),(C_{q/p},\infty_{C}))$  is the same as a $\QQ$-linear ring homomorphism $\phi^{*}: \Gamma(C_{q/p}-\{\infty_{C}\},\calO_{C})=\QQ[t^{p}, t^{q}]\to A[t]$ such that if $f\in \QQ[t^{p}, t^{q}]$ has degree $n>0$ in $t$, then $\phi^{*}(f)\in A[t]$ also has degree $n$ with leading coefficient in $A^{\times}$. The base point $\phi_{0}$ corresponds to the tautological embedding $\QQ[t^{p}, t^{q}]\incl \QQ[t]$.

Now let $M^{rig}_{1}(\PP^{1},C_{q/p})\subset M_1((\PP^1,\infty),(C_{q/p},\infty_{C}))$ be the subscheme whose $R$-points consist of $\phi:\PP^{1}_{R}\to C_{q/p}$ as above such that, in addition, $\phi^{*}(f)-f$ has degree $\leq n-2$ for if $\deg_{t}(f)=n>0$. Then the reduced structure of $M^{rig}_1(\PP^1,C_{q/p})$ consists of one point $\phi_{0}$. The ring of regular functions
\begin{equation*}
\cO_{q/p}:=\cO(M^{rig}_1(\PP^1,C_{q/p}))
\end{equation*}
is a local artinian ring with residue field $\QQ$.

The $\Gm(q/p)$-action on $C_{q/p}$ induces one on $M_1^{rig}(\PP^1,C_{q/p})$, hence a grading on its coordinate ring $\cO_{q/p}$. We denote this grading by $\cO_{q/p}=\oplus_{j\in\ZZ}\cO_{q/p}[j]$. 

An $R$-point $\phi$ of $M^{rig}_{1}(\PP^{1},C_{q/p})$ is determined by the elements $\phi^{*}(t^{p})$ and $\phi^{*}(t^{q})$. We write
\begin{eqnarray*}
\phi^{*}(t^{p})=t^{p}+\sum_{i=2}^{p}e_{i}t^{p-i},\\
\phi^{*}(t^{q})=t^{q}+\sum_{i=2}^{p}f_{i}t^{q-i}.
\end{eqnarray*}
Then $e_{i},f_{i}\in R$ vary with $\phi$. Therefore $e_{i}$ and $f_{i}$ can be viewed as elements in $\cO_{q/p}$ homogeneous of degree $i$.

We recall the following presentation of the coordinate ring $\cO_{q/p}$.

\begin{prop}[Beauville, see {\cite[Section G]{FGS}}]\label{l:Oqp}
\begin{enumerate}
\item There is an isomorphism of graded rings
\begin{equation*}
\cO_{q/p}\cong \QQ[e_2,\dots,e_{p},f_2,\dots,f_{q}]/I_{\cO}
\end{equation*}
where $I_{\cO}$ is the homogeneous ideal generated by the coefficients of $w$ of 
$$(1+e_2 w^{2}+\dots+e_p w^p)^q=(1+f_2 w^2+\dots+f_q w^q)^p.$$
\item The $\QQ$-algebra $\cO_{q/p}$ is a complete intersection, and
$$ \dim_{\QQ}\cO_{q/p}=\frac{1}{p+q}\binom{p+q}{p}.$$

\item There is an isomorphism of graded rings 
\begin{equation*}
\cO_{q/p}\cong \QQ[e_2,\dots,e_{p}]/(g_{q+1},\cdots,g_{p+q-1})
\end{equation*}
where $g_{q+i}$ is the coefficient of $w^{q+i}$ of the Taylor expansion of 
$$ (1+e_2 w^{2}+\dots+e_p w^p)^{q/p}$$
at $w=0$. 
\end{enumerate}
\end{prop}
\begin{proof}
(1) As $t^{p}$ and $t^{q}$ generate $\QQ[t^{p},t^{q}]$, an $R$-point $\phi$ of $M^{rig}_{1}(\PP^{1},C_{q/p})$ is determined by the coefficients of $\phi^{*}(t^{p})$ and $\phi^{*}(t^{q})$. Therefore $e_{i}, 2\le i\le p$ and $f_{j}, 2\le j\le q$ generate $\cO_{q/p}$.

The only relation that $\phi^{*}(t^{p})$ and $\phi^{*}(t^{q})$ need to satisfy is $(\phi^{*}(t^{p}))^{q}=\phi^{*}(t^{pq})=(\phi^{*}(t^{q}))^{p}$. Expanding it and extracting homogeneous parts of it gives the defining equations for $\cO_{q/p}$, which generate exactly the ideal $I_{\cO}$.

(2) Let $e(w)=1+e_{2}w+\cdots+e_{p}w^{p}$; $f(w)=1+f_{2}w+\cdots+f_{q}w^{q}$. If we take logarithmic derivatives of both parts of the equation $e(w)^{q}=f(w)^{p}$, we get an equivalent equation
\begin{equation*}
qe'(w)f(w)=pe(w)f'(w).
\end{equation*}
Both sides of the above are polynomials in $w$ of degree $p+q-1$ and the leading terms are the same, and the constant terms are zero. Hence the ideal $I_{\cO}$ is generated by the coefficients of $qe'(w)f(w)-pe(w)f'(w)$ in front of $w,w^{2}\cdots, w^{p+q-2}$, which are homogeneous of degrees $2,3,\cdots, p+q-1$. In particular,  $\cO$ is a complete intersection with generators in degrees $2,3,\cdots, p; 2,\cdots, q$, and  generators of the defining ideal in degrees $2,3,\dots,p+q-1$. This implies the Hilbert series
\begin{equation*}
\sum_{j\geq0}\dim_{\QQ}\cO_{q/p}[j]T^{j}=\frac{(1-T^{q+1})\cdots (1-T^{q+p-1})}{(1-T^{2})\cdots (1-T^{p})}
\end{equation*}
Evaluating at $T=1$ we get $\dim_{\QQ}\cO_{q/p}=\frac{1}{p+q}\binom{p+q}{p}$.

(3) Let $g_{i}\in\QQ[e_{2},\cdots, e_{p}]$ be the  coefficient of $e(w)^{q/p}$ in front of $w^{i}$. The desired equation $e(w)^{q}=f(w)^{p}$ is equivalent to saying that $e(w)^{q/p}$ is a polynomial of degree $\le q$. Therefore we have
\begin{equation*}
\cO_{q/p}\cong \QQ[e_{2},\cdots, e_{p}]/(g_{i}, \forall i>q).
\end{equation*}
Now we show that  the vanishing of $g_{q+1},\cdots,g_{p+q-1}$ implies the vanishing of all $g_{i}$ ($i>q$). Indeed, suppose $\phi: \QQ[e_{2},\cdots, e_{p}]/(g_{q+1},\cdots, g_{p+q-1})\to A$ is any ring homomorphism. We write $\ov e(w), \ov g_{i}$ for the images of $e(w)$ and $g_{i}$ in $A$. Write $\ov e(w)^{q/p}=\ov f(w)+\ov h(w)$, where  $\ov f(w)=1+\ov g_{2}w^{2}+\cdots+\ov g_{q}w^{q}$ and $\ov h(w)=\sum_{j>q}\ov g_{j}w^{j}$. Then $\ov e(w)^{q}=(\ov f(w)+\ov g(w))^{p}$, which implies $q\ov e'(w)(\ov f(w)+\ov h(w))=p\ov e(w)(\ov f'(w)+\ov h'(w))\in A[[w]]$ by taking logarithmic derivatives. Equivalently,
\begin{equation*}
q\ov e'(w)\ov f(w)-p\ov e(w)\ov f'(w)=-q\ov e'(w)\ov h(w)+p\ov e(w)\ov h'(w).
\end{equation*} 
Now the LHS has degree $\le p+q-2$ while all terms on the RHS have degree $\ge p+q-1$. Therefore both sides must vanish. The vanishing of the LHS implies $\ov f(w)=\ov e(w)^{q/p}$ hence $\ov h(w)=0$, i.e., $\ov g_{i}=0$ for all $i\ge p+q$.\end{proof}

\subsection{The geometry of $J_{q/p}$} Recall $J_{q/p}=\ov\Jac(C_{q/p})$ is the compactified Jacobian of $C_{q/p}$ rigidified at $\infty_{C}$. Below we introduce a related moduli space which will be turn out to be an affine Springer fiber. 

\subsubsection{The moduli spaces $X_{q/p}$}\label{Xqp}  Consider the functor $X_{q/p}$ on commutative $\CC$-algebras whose value on $R$ is the set of $R\tl{t^{p},t^{q}}$-submodules $M\subset R\lr{t}$ such that there exists $N>0$ such that
\begin{enumerate}
\item $t^{N}R\tl{t}\subset M\subset t^{-N}R\tl{t}$;
\item $M/t^{N}R\tl{t}$ is a projective $R$-module of rank equal to the rank of $R\tl{t^{p},t^{q}}/t^{N}R\tl{t}$.
\end{enumerate}
Then $X_{q/p}$ is represented by an ind-scheme. The scaling action $\Gm$ on $t$ induces an $R$-linear ring automorphism of $R\tl{t^{p},t^{q}}$ for any $R$, hence we get a $\Gm$-action on $X_{q/p}$. We also denote this $\Gm$ by $\Gm(q/p)$.

\subsubsection{Relation between $X_{q/p}$ and $J_{q/p}$}\label{sss:relXJ} We will relate $X_{q/p}$ and $J_{q/p}$ via a third space $\wt J_{q/p}$ and two maps
\begin{equation*}
\xymatrix{       & \wt J_{q/p}\ar[dl]_{\l}\ar[dr]^{\om}    \\
X_{q/p} & & J_{q/p}}
\end{equation*}

Let $\wt J_{q/p}$ be the sheafification of the functor whose $R$-points consists of pairs $(\cL, \t)$ where $\cL$ is a torsion-free rank one coherent sheaf on $C_{q/p}\otimes R$ with the same fiberwise degree as the structure sheaf of $C_{q/p}$, and $\t$ is a trivialization of $\cL$ over $(C_{q/p}-\{0\})\otimes R$. We have a forgetful map $\om: \wt J_{q/p}\to J_{q/p}$ by restricting the trivialization $\t$ to $\infty_{C}$ to get the rigidification. Then $\om$ is a homeomorphism because when $R$ is a field, a trivialization of a line bundle on $(C_{q/p}-\{0\})\otimes R\cong \AA^{1}_{R}$ is the same as a trivialization of it at the point $\infty_{C,R}$. On the other hand, we have an isomorphism $\l: \wt J_{q/p}\to X_{q/p}$ which on the level of $R$-points sends $(\cL,\t)$ to the restriction $\cL|_{\Spec R\lr{t}}$ of $\cL$ in the punctured formal disk around $0_{C,R}$, which is embedded into $R\lr{t}$ using the trivialization $\t$. Both maps $\l$ and $\om$ are $\Gm(q/p)$-equivariant (the $\Gm(q/p)$-action on $\wt J_{q/p}$ is induced from that on $C_{q/p}$). In conclusion, we have a $\Gm(q/p)$-equivariant homeomorphism \footnote{The map $\mu$ is not an isomorphism for two reasons: first, the ind-scheme $X_{q/p}$ is highly non-reduced; second, even if we pass to the reduced structure of $X_{q/p}$, it may be non-isomorphic to $J_{q/p}$. For example, when $(p,q)=(2,3)$, we have $X_{q/p}^{\red}\cong \PP^{1}$ while $J_{q/p}$ is isomorphic to the curve $C_{q/p}$ with a cusp.}
\begin{equation}\label{XJ}
\mu: X_{q/p}\to J_{q/p}.
\end{equation}

\subsubsection{Torus fixed points in $X_{q/p}$}\label{FixedPoints} The $\Gm(q/p)$-fixed points on $X_{q/p}$ are indexed by those $M\subset \CC\lr{t}$ topologically spanned by powers of $t$. Let $\jiao{p,q}\subset \ZZ$ be the submonoid generated by $p$ and $q$. For any such $M\in X_{q/p}^{\Gm(q/p)}$ let $\s_{M}=\{n\in\ZZ|t^{n}\in M\}$, then $\s_{M}\subset\ZZ$ is a $\jiao{p,q}$-module. The condition on the volume of $M$ implies that $\s_{M}-\s_{M}\cap \jiao{p,q}$ has the same cardinality as $\jiao{p,q}-\s_{M}\cap \jiao{p,q}$.
Let $\Sigma_{q/p}$ be the set of $\jiao{p,q}$-submodule of $\ZZ$ satisfying the conditon that $\#(\s-\s\cap \jiao{p,q})=\#(\jiao{p,q}-\s\cap \jiao{p,q})$. Then for any $\s\in \Sigma_{q/p}$, let $M_{\s}$ be the $\CC\tl{t^{p}, t^{q}}$-submodule of $\CC\lr{t}$ generated by $t^{n}, n\in \s$. Then $M_{\s}$ is a $\Gm(q/p)$-fixed point of $X_{q/p}$. Therefore, the maps $M\mapsto \s_{M}$ and $\s\mapsto M_{\s}$ are inverse to each other and they give a bijection
\begin{equation*}
X^{\Gm(q/p)}_{q/p}\bij \Sigma_{q/p}.
\end{equation*}

For each $\s\in\Sigma_{q/p}$ and $0\leq i\leq p$, let $a_{i}$ be the smallest element in $\s$ congruent to $i$ modulo $p$; we call $\{a_{0},\cdots,a_{p-1}\}$ the {\em $p$-basis} for $\s$. Similarly we define the {\em $q$-basis} $\{b_{0},\cdots,b_{q-1}\}$ for $\s$. The condition $\#(\s-\s\cap \jiao{p,q})=\#(\jiao{p,q}-\s\cap \jiao{p,q})$ implies
\begin{equation}\label{absum}
\sum_{i=0}^{p-1}a_{i}=pq(p-1)/2; \quad \sum_{j=0}^{q-1}b_{j}=pq(q-1)/2.
\end{equation}

By \cite[Theorem 2(i)]{LS} (our set $\Sigma_{q/p}$ is in bijection with the set $Z_{p,q}$ defined in \cite[Section 4]{LS}), we have
\begin{equation}\label{fixed pts}
\#X^{\Gm(q/p)}_{q/p}=\#\Sigma_{q/p}=\frac{1}{p+q}\binom{p+q}{p}.
\end{equation}

\subsubsection{}\label{sss:paving} By \cite[Theorem 2(ii)]{LS}, for each $\Gm(q/p)$-fixed point of $X_{q/p}$ indexed by $\s\in\Sigma_{q/p}$, the corresponding attracting subset $C_{\s}$ is an affine space. In particular, $X_{q/p}$ has a partition into affine spaces, hence the cohomology of $X_{q/p}$ (or $J_{q/p}$) is concentrated in even degrees, pure and Tate.

\subsection{Relation with affine Springer fibers}
\sss{$X_{q/p}$ as an affine Springer fiber}\label{Xqp ASF} Let $G=\SL_{p}$, $\frg=\sl_{p}$ and let $\Gr_{G}=G\lr{x}/G\tl{x}$ be the affine Grassmannian of $G$. Let $\g_{q/p}\in\frg\lr{x}$ which under the standard basis $v_{0},\cdots, v_{p-1}$ is given by $\g_{q/p}(v_{i})=v_{i+1}$ for $0\leq i\leq p-2$ and $\g_{q/p}(v_{p-1})=x^{q}v_{0}$. The characteristic polynomial of $\g_{q/p}$ is $P(y)=y^{p}-x^{q}$. We denote by $\SG$ the affine Springer fiber associated to $\g_{q/p}$, whose $\CC$-points are
\begin{equation*}
\SG(\CC)=\{ g \in G(\CC\lr{x})/G(\CC\tl{x})| \Ad(g^{-1})(\g_{q/p})\in \frg\tl{x}\}.
\end{equation*}

\begin{lemma}\label{l:SGX} There is a canonical isomorphism
\begin{equation*}
\SG\cong X_{q/p}.
\end{equation*}
Combining with \eqref{XJ} we have a canonical homeomorphism
\begin{equation}\label{SGJ}
\SG\to J_{q/p}.
\end{equation}
\end{lemma}
\begin{proof}
Let $V=\CC\lr{x} v_{0}\oplus \cdots\oplus\CC\lr{x} v_{p-1}$ be the standard representation of $G$ tensored with $\CC\lr{x}$. Let $\L^{\dagger}=\CC\tl{x} v_{0}\oplus \cdots\oplus\CC\tl{x} v_{p-1}\subset V$. Recall that $\Gr_{G}$ classifies $\CC\tl{x}$-submodules of $V$ of rank $n$. As a subspace of $\Gr_{G}$, $\SG$ classifies lattices $\L\subset V$ such that $\g_{q/p}\L\subset \L$ and that $\dim(\L/\L\cap\L^{\dagger})=\dim(\L^{\dagger}/\L\cap\L^{\dagger})$. Consider the isomorphism $V\isom \CC\lr{t}$ sending $x^{j}v_{i}$ to $t^{pj+qi}$, for $j\in\ZZ, 0\leq i\leq p-1$. Under this isomorphism, the action of $\g_{q/p}$  on $V$ becomes the multiplication by $t^{q}$ on $\CC\lr{t}$. Then a $\CC\tl{x}$-lattice $\L$ in $V$ stable under $\g_{q/p}$ is the same as a $\CC\tl{t^{p},t^{q}}$-lattice $M$ in $\CC\lr{t}$. Moreover, the condition $\dim(\L/\L\cap\L^{\dagger})=\dim(\L^{\dagger}/\L\cap\L^{\dagger})$ is equivalent to the condition (2) in Section~\ref{Xqp}. Therefore $\SG$ is canonically isomorphic to $X_{q/p}$.  
\end{proof}

The isomorphism \eqref{SGJ} will enable us to use results from \cite{OY} to study the cohomology of $J_{q/p}$.

\sss{Affine Springer fiber in the affine flag variety}\label{sss:aff flag}
A variant of $\SG$ is the affine Springer fiber $\Sp_{q/p}$ of $\g_{q/p}$ in the affine flag variety $\Fl_{G}$ of $G=\SL_{p}$. It classifies chains of $\CC\tl{x}$-lattices $\L_{\bullet}: \L_{0}=x\L_{p}\subset \L_{1}\subset \cdots \subset \L_{p-1}\subset \L_{p}$ in $V$ such that each $\L_{i}$ is stable under $\g_{q/p}$, $\L_{p}\in\SG$ and $\dim(\L_{i}/\L_{i-1})=1$ for $1\le i\le p$. By identifying $V$ with $\CC\lr{t}$, we may $\Sp_{q/p}$ also classifies chains of $\CC\tl{t^{p},t^{q}}$-lattices $\L_{\bullet}: \L_{0}=t^{p}\L_{p}\subset \L_{1}\subset \cdots \subset \L_{p-1}\subset \L_{p}$ inside $\CC\lr{t}$ such that $\L_{p}\in X_{q/p}$.

For $1\leq i\leq p$, $\Sp_{q/p}$ carries a $\Gm(q/p)$-equivariant line bundle $\cL_{i}$ whose fiber at $\L_{\bullet}$ is $\L_{i}/\L_{i-1}$. 

We may analyze the $\Gm(q/p)$-fixed points on $\Sp_{q/p}$ as in Section~\ref{FixedPoints}. Let $\wt\Sigma_{q/p}$ be the set of ordered $p$-tuples $(d_{1},d_{2},\cdots, d_{p})\in\ZZ^{p}$ such that 
\begin{itemize}
\item The residue classes of $\{d_{1},\cdots, d_{p}\}$ modulo $p$ are distinct (hence exhausting all residue classes mod $p$).
\item For each $1\le i\le p$, $d_{i}+q$ is in the $p\ZZ_{\ge0}$-module generated by $d_{i-1},\cdots, d_{1},d_{p}+p,\cdots, d_{i}+p$.
\end{itemize}
For $\un d=(d_{1},d_{2},\cdots, d_{p})\in \wt\Sigma_{q/p}$ and $1\le i\le p$, we define a $(p,q)$-module $\s_{i}(\un d)$ generated by $d_{i-1},\cdots, d_{1},d_{p}+p,\cdots, d_{i}+p$. Let $\L_{i}(\un d)$ be the $\CC\tl{t^{p},t^{q}}$-submodule of $\CC\lr{t}$ generated by $t^{n}$ for $n\in \s_{i}(\un d)$. Then $\L_{\bullet}(\un d):\L_{0}(\un d)=t^{p}\L_{p}(\un d)\subset \L_{1}(\un d)\subset \cdots \subset \L_{p-1}(\un d)\subset \L_{p}(\un d)$ is a $\Gm(q/p)$-fixed point of $\Sp_{q/p}$. The map $\un d\mapsto \L_{\bullet}(\un d)$ gives a bijection
\begin{equation*}
\Sp_{q/p}^{\Gm(q/p)}\bij\wt\Sigma_{q/p}.
\end{equation*}

\section{The cohomology rings of $J_{q/p}$ and $\Sp_{q/p}$: generators}\label{s:gen}
The goal of this section is to prove Theorems~\ref{th:gen} and \ref{th:par gen}. The argument involves passing to certain moduli spaces of global flavor, namely moduli space of stable Higgs bundles on the weighted projective line.

\subsection{Generators for the cohomology ring of the moduli of stable Higgs bundles}\label{ss:glob gen} In this subsection we consider the Hitchin moduli space $M_{r,d}$ of stable Higgs bundles over a curve $X$. Hausel-Thaddeus \cite{HT} proved that the K\"unneth components of the Chern classes of the universal bundle generate the cohomology ring of $M_{r,d}$. This was shown by a different method by Markman \cite{Mark}, using a result of Beauville \cite{Beau}. 

For our application, we need the case where $X$ is a weighted projective line, which is not covered by the literature. In the following we work with Deligne-Mumford curves and extend Markman's argument for the generation statement to such generality. The generation result will then be used to prove the generation result for $J_{q/p}$ and $\Sp_{q/p}$ in the next subsection.

\subsubsection{Moduli of Higgs bundles}\label{sss:Moduli Higgs} Let $X$ be an irreducible smooth proper Deligne-Mumford curve. Let $\calL$ be a line bundle over $X$. An $\calL$-valued Higgs bundle over $X$ is a pair $(\calE,\varphi)$ where $\calE$ is a vector bundle over $X$ and $\varphi:\calE\to\calE\otimes\calL$. Now fix a line bundle $\D$ on $X$. For $r\in\ZZ_{\ge1}$ and $d\in \QQ$, let $\cM_{r,\D}(X,\cL)$ be the moduli stack of triples $(\cE, \ph, \iota)$ where $(\cE,\ph)$ is an $\cL$-valued Higgs bundles on $X$ of rank $r$, and $\iota$ is an isomorphism $\det(\cE)\cong \D$. 

We have the notion of slope stability for Higgs bundles by considering Higgs subbundles (those preserved by $\ph$) instead of all subbundles.  Therefore we get an open substack $\cM^{s}_{r,\D}(X,\cL)\subset \cM_{r,d}(X,\cL)$ consisting of stable $\cL$-valued Higgs bundles on $X$. 

Nitsure's construction of the coarse moduli spaces of stable (resp. semistable) Higgs bundles over an algebraic curve in \cite{Ni} works in the case of a Deligne-Mumford curve. In particular, we get a quasi-projective moduli space $M^{s}_{r,\D}(X,\cL)$ of stable  $\cL$-valued Higgs bundles on $X$ of rank $r$ and determinant $\D$. We remark that from Nitsure's construction via geometric invariant theory,  $M^{s}_{r,\D}(X,\cL)$ is the coarse moduli space of $\cM^{s}_{r,\D}(X,\cL)$, and the natural map $\om: \cM^{s}_{r,\D}(X,\cL)\to M^{s}_{r,\D}(X,\cL)$ is a $\mu_{r}$-gerbe.  In particular, $\cM^{s}_{r,\D}(X,\cL)$ is a Deligne-Mumford stack all of whose points have automorphism group $\mu_{r}$ (acting by scalar multiplication on Higgs bundles).

\subsubsection{Moduli of parabolic Higgs bundles}\label{sss:par Hit} Fix a point $x_{0}\in X$. An  $\calL$-valued Higgs bundle over $X$ with a parabolic structure at $x_{0}$ is a triple $(\calE,\ph, F_{\bullet}\cE_{x_{0}})$ where $(\cE,\ph)$ is an $\calL$-valued Higgs bundle and  $0=F_{0}\cE_{x_{0}}\subset F_{1}\cE_{x_{0}}\subset \cdots\subset F_{r-1}\cE_{x_{0}}\subset F_{r}\cE_{x_{0}}=\cE_{x_{0}}$ is a full flag of the fiber of $\cE$ at $x_{0}$ (where $r$ is the rank of $\cE$) such that $\ph_{x_{0}}(F_{i}\cE_{x_{0}})\subset F_{i}\cE_{x_{0}}\otimes\cL_{x_{0}}$ for all $i$. For the a line bundle $\D$ on $X$,  we define $\cM^{\pb}_{r,\D,x_{0}}(X,\cL)$ to be the moduli stack of quadruples $(\calE,\ph, F_{\bullet}\cE_{x_{0}}, \iota)$ where $(\cE, \ph,F_{\bullet}\cE_{x_{0}})$ is an $\calL$-valued Higgs bundles with parabolic structure at $x_{0}$ over $X$ of rank $r$,  and $\iota: \det(\cE)\cong\D$. We often abbreviate $\cM^{\pb}_{r,\D,x_{0}}(X,\cL)$ by $\cM^{\pb}_{r,\D}(X,\cL)$ when $x_{0}$ is fixed in the context. There is a forgetful map
\begin{equation*}
\nu: \cM^{\pb}_{r,\D}(X,\cL)\to \cM_{r,\D}(X,\cL).
\end{equation*}
The fiber of $\nu$ over $(\cE,\ph,\iota)$ consists of full flags of the vector space $\cE_{x_{0}}$ stable under $\ph_{x_{0}}$. Therefore the fibers of $\nu$ are Springer fibers for the group $\GL_r$. In particular, $\nu$ is representable and proper.
 
We define $\cM^{\pb,s}_{r,\D}(X,\cL)$ to be the preimage of $\cM^{s}_{r,\D}(X,\cL)$ under $\nu$. Similarly we have the coarse moduli space $M^{\pb,s}_{r,\D}(X,\cL)$ such that the natural map $\cM^{\pb,s}_{r,\D}(X,\cL)\to M^{\pb,s}_{r,\D}(X,\cL)$ is a $\mu_{r}$-gerbe. 

\begin{prop}\label{p:sm}
\begin{enumerate}
\item Suppose $\deg\cL>\deg\om_{X}$, then the stack $\cM^{s}_{r,\D}(X, \cL)$ and the scheme $M^{s}_{r,\D}(X,\cL)$ are smooth. 
\item Suppose $\deg\cL\ge\deg\om_{X}+1$, then the stack $\cM^{\pb, s}_{r,\D}(X, \cL)$ and the scheme $M^{\pb, s}_{r,\D}(X,\cL)$ are smooth. 
\end{enumerate}
\end{prop}
\begin{proof} 

(1) Since $\cM^{s}_{r,\D}(X, \cL)$ is a $\mu_{r}$-gerbe over $M^{s}_{r,\D}(X,\cL)$, it suffices to show that $\cM^{s}_{r,\D}(X, \cL)$ is smooth. We will show that the semistable locus  $\cM^{ss}_{r,\D}(X, \cL)$ is smooth, i.e.,  $\cM_{r,\D}(X, \cL)$ is smooth at any semistable point $(\cE,\ph)$. The argument is borrowed from Nitsure \cite[Proof of Proposition 7.1]{Ni}. The obstruction to smoothness lies in $\cohog{2}{X, \cK}$ where $\cK$ is the complex $\un\End_{0}(\cE)\xr{[\ph,-]} \un\End(\cE)\otimes\cL$ placed in degrees $0$ and $1$, and $\un\End_{0}(\cE)$ is the sheaf of trace-free endomorphisms of $\cE$. By Serre duality, $\cohog{2}{X, \cK}$ is dual to $\cohog{0}{X,\cK'}$ where $\cK'$ is $\un\End(\cE)\otimes\cL^{-1}\otimes\om_{X}\xr{[\ph,]} \un\End_{0}(\cE)\otimes\om_{X}$ placed in degrees $0$ and $1$. Now $\cohog{0}{X,\cK'}$ consists of maps $\psi: \cE\to \cE\otimes\cL^{-1}\otimes\om_{X}$ commuting with $\ph$. However $(\cE\otimes\cL^{-1}\otimes\om_{X},\ph\otimes1\otimes 1)$ is a stable Higgs bundle with slope $\mu(\cE)-\deg\cL+\deg\om_{X}<\mu(\cE)$, such $\psi$ must be zero. Therefore the obstruction vanishes, hence $\cM_{r,\D}(X, \cL)$ is smooth at $(\cE,\ph)$.

(2) Again it suffices to show that $\cM^{\pb, s}_{r,\D}(X, \cL)$ is smooth, i.e., $\cM^{\pb}_{r,\D}(X, \cL)$ is smooth at any point $(\cE,\ph, F_{\bullet}\cE_{x_{0}})$ whenever $(\cE,\ph)$ is stable. Similar calculation shows that the obstruction to smoothness lies in $\cohog{2}{X, \cP}$ where $\cK^{\pb}$ is the complex $\un\End'_{0}(\cE)\xr{[\ph,]}\un\End'(\cE)\otimes\cL$ (in degrees $0$ and $1$), and $\un\End'_{0}(\cE)\subset\un\End_{0}(\cE)$ is the subsheaf preserving the flag $F_{\bullet}\cE_{x_{0}}$. By Serre duality, $\cohog{2}{X, \cK^{\pb}}$ is dual to $\cohog{0}{X, \cP'}$ where $\cP'$ is the complex $\un\End''(\cE)\otimes\cL^{-1}\otimes \om_{X}(x_{0})\xr{[\ph,-]}\un\End''_{0}(\cE)\otimes\om_{X}(x_{0})$, and $\un\End''(\cE)\subset \un\End'(\cE)$ (resp. $\un\End''_{0}(\cE)\subset \un\End'_{0}(\cE)$) is the subsheaf of those local endomorphisms of $\cE$ sending $F_{i}\cE_{x_{0}}$ to $F_{i-1}\cE_{x_{0}}$ for all $1\le i\le r$. An element $\psi\in \cohog{0}{X, \cP'}$ is then a map $\psi: \cE\to \cE\otimes\cL^{-1}\otimes \om_{X}(x_{0})$ commuting with $\ph$ and sending $F_{i}\cE_{x_{0}}$ to $F_{i-1}\cE_{x_{0}}$. If $\deg\cL>\deg\om_{X}+1$, the same argument as in (1) shows that $\psi=0$. If  $\deg\cL=\deg\om_{X}+1$, the argument there forces any nonzero $\psi$ to be an isomorphism, which contradicts the fact that $\psi_{x_{0}}$ is nilpotent. Again we have shown that the obstruction vanishes.
\end{proof}

Let $(\cE^{\univ},\ph^{\univ})$ be the universal Higgs bundle over $\cM_{r,\D}(X,\cL)\times X$.  We consider the Chern classes $c_{i}(\Eun)\in\cohog{2i}{\cM^{s}_{r,\D}(X,\cL)\times X}$, $i=1,\cdots,r$. We may talk about the K\"unneth components of $c_{i}(\Eun)$ in $\cohog{*}{\cM^{s}_{r,\D}(X,\cL)}$. 

For a scheme $M$ over $\CC$ and each $n\in\ZZ$, define
\begin{equation*}
\cohog{n}{M}_{\pure}:=\Gr^{W}_{n}\cohog{n}{M}
\end{equation*}
under the mixed Hodge structure on $\cohog{n}{M}$. When $M$ is smooth, $\Gr^{W}_{n}\cohog{n}{M}=W_{n}\cohog{n}{M}$ is a subspace of $\cohog{n}{M}$, hence the direct sum
\begin{equation*}
\cohog{*}{M}_{\pure}=\bigoplus_{n\in\ZZ}\cohog{n}{M}_{\pure}
\end{equation*}
is a subalgebra of $\cohog{*}{M}$.  Chern classes of vector bundles as well as cycle classes on a smooth scheme $M$ always lie in $\cohog{*}{M}_{\pure}$. 

\begin{theorem}[Hausel-Thaddeus \cite{HT}, Markman \cite{Mark}]\label{th:Hit gen} Suppose $\deg\cL>\deg\om_{X}$. Let $\cM'\subset \cM^{s}_{r,\D}(X,\cL)$ be any open substack such that the universal bundle descends to a bundle $\un\cE^{\univ}$ on $M'\times X$, where $M'$ is the coarse moduli space of $\cM'$. Then the pure part of the cohomology ring $\cohog{*}{M'}_{\pure}$ is generated by the K\"unneth components of the universal Chern classes $c_{i}(\un\cE^{\univ})\in\cohog{2i}{M'\times X}$, $1\leq i\leq r$, as a $\QQ$-algebra.
\end{theorem} 
\begin{proof} Whether or not $\Eun$ descends to the coarse moduli space, the bundle $\un\End(\Eun)$ always descends for the pointwise automorphism group $\mu_{r}$ acts trivially on it. When $\Eun$ descends to  bundle $\un\cE^{\univ}$ on $M'\times X$, the universal Higgs field $\ph^{\univ}$ also descends to $M'\times X$ since it is a section of $\un\End(\un\cE^{\univ})\otimes \cL$ on $M'\times X$.

We have a two term complex $\cH^{\univ}$ on $M'\times M'\times X$ concentrated in degrees $0$ and $1$:
\begin{equation*}
\Hun=[\un\Hom(\pi^{*}_{13}\un{\cE}^{\univ}, \pi^{*}_{23}\un{\cE}^{\univ}) \xr{\d}\un\Hom(\pi^{*}_{13}\un{\cE}^{\univ}, \pi^{*}_{23}\un{\cE}^{\univ})\otimes\pi_{3}^{*}\cL]
\end{equation*}
Here $\pi_{12}:M'\times M'\times X\to M'\times X$ is the projection to the first two factors, etc. The map $\d$ is given by
\begin{equation*}
h\mapsto (h\otimes\id_{\calL})\circ\pi^{*}_{13}\ph^{\univ}-(\pi^{*}_{23}\ph^{\univ}\otimes\id_{\cL})\circ h.
\end{equation*}
For two Higgs bundles $(\calE,\varphi)$ and $(\calF,\psi)$ in $M'$, the restriction of $\Hun$ to $\{(\calE,\varphi)\}\times\{(\calF,\psi)\}\times X$ is the complex on $X$
\begin{equation*}
\calH_{(\calE,\varphi),(\calF,\psi)}:=[\un\Hom(\calE,\calF)\xrightarrow{\delta(\varphi,\psi)}\un\Hom(\calE,\calF)\otimes\calL]
\end{equation*}
where  the map $\delta(\varphi,\psi)$ is $h\mapsto (h\otimes\id_{\calL})\circ\varphi-(\psi\otimes\id_{\cL})\circ h$.  

\begin{claim} For any two Higgs bundles $(\calE,\varphi)$ and $(\calF,\psi)$ in $M'$, $\cohog{i}{X,\calH_{(\calE,\varphi),(\calF,\psi)}}=0$ for $i\geq 2$. 
\end{claim}
\begin{proof}[Proof of Claim]  We abbreviate $\calH_{(\calE,\varphi),(\calF,\psi)}$ by $\calH$. Clearly $\cohog{i}{X,\calH}=0$ for $i\geq 3$ because $\calH$ is in degrees $\leq1$ and $X$ has dimension $1$. By Serre duality, we have $\cohog{2}{X,\calH}\cong\cohog{0}{X,\calH'}^{*}$, where $\calH'$ is the two term complex in degrees $0$ and $1$
\begin{equation*}
\calH'=[\un\Hom(\calF,\calE)\otimes\calL^{-1}\otimes\om_{X}\xr{\delta(\psi,\ph)}\un\Hom(\calF,\calE)\otimes\om_{X}].
\end{equation*}
An element $h\in \cohog{0}{X,\calH'}$ is then a map $h: \cF\to \cE\otimes\cL^{-1}\otimes\om_{X}$ that intertwines $\psi$ and $\ph\otimes 1\otimes 1$, i.e., a Higgs bundle map. Since $\cE\otimes\cL^{-1}\otimes\om_{X}$ is stable with slope $\mu(\cE)-\deg \cL+\deg\om_{X}<\mu(\cE)=\mu(\cF)$, $h$ must be zero. Hence  $\cohog{2}{X,\calH}=0$ by duality.
\end{proof}

\begin{claim}[Beauville \cite{Beau}]\label{claim:diag}  Let $m=\dim M'$. Then $c_{m}(\bR\pi_{12,*}\Hun)\in\cohog{2m}{M'\times M'}$ is the cycle class of a cycle supported on the diagonal $\D(M')$, with positive coefficient on each irreducible component of $\D(M')$.
\end{claim}
\begin{proof}[Proof of Claim] By the previous Claim,   $\bR\pi_{12,*}\Hun$ is quasi-isomorphic to a two-term complex of vector bundles $K^{0}\xrightarrow{u}K^{1}$ on $M'\times M'$. For a point $\xi=((\calE,\varphi),(\calF,\psi))\in M'\times M'$, we have $\ker(K^{0}(\xi)\xrightarrow{u(\xi)}K^{1}(\xi))=\Hom_{X}((\calE,\varphi), (\calF,\psi))$ which is nonzero if and only if $(\calE,\varphi)\cong(\calF,\psi)$. Therefore the degeneracy locus  of $u$ is exactly $\Delta(M')$. This means $u$ is injective. Hence $\bR\pi_{12,*}\cH^{\univ}$ is quasi-isomorphic to a coherent sheaf $\cQ$ on $M'\times M'$ which is locally has a two term resolution $K^{0}\xrightarrow{u}K^{1}$ by locally free sheaves. Hence the degeneracy locus of $u$ (i.e., $\Delta(M')$) has the expected codimension $m=\rk K^{1}-\rk K^{0}+1$. Thom-Porteous formula (see \cite[Theorem 14.4(a)(b)]{Fulton}) then shows that $c_{m}(\bR\pi_{12,*}\Hun)$ is the class of a cycle supported on $\D(M')$ with positive coefficients on each irreducible component of $\Delta(M')$.
\end{proof}

Let $\cohog{*}{M'}_{\ch}$ be the {\em subalgebra} generated by the K\"unneth components of $c_{i}(\un{\cE}^{\univ}), i=1,\cdots, r$. Our goal is to show that $\cohog{*}{M'}_{\ch}=\cohog{*}{M'}_{\pure}$. Let $\cohog{*}{M'}_{\Delta}$ be the {\em linear subspace} spanned by the K\"unneth components of the cycle class $cl(\Delta(M'))\in\cohog{2m}{M'\times M'}$. 

We first show that $\cohog{*}{M'}_{\Delta}\subset \cohog{*}{M'}_{\ch}$. By the Grothendieck-Riemann-Roch formula, the Chern character of $\bR\pi_{12,*}\Hun$ is expressed as
\begin{equation}\label{RR}
\ch(\bR\pi_{12,*}\Hun)=\pi_{12,*}(\ch(\Hun)\cdot\pi^{*}_{X}\Todd_{X}).
\end{equation}
Clearly, $\ch(\Hun)$ can be expressed as a polynomial of $\pi_{13}^{*}c_{i}(\un{\cE}^{\univ})$ and $\pi_{23}^{*}c_{i}(\un{\cE}^{\univ})$. Therefore \eqref{RR} implies that $\ch(\bR\pi_{12,*}\Hun)\in\cohog{*}{M'\times M'}$ can be expressed as  $\sum_{j}\alpha_{j}\otimes\beta_{j}$ under the K\"unneth decomposition, where $\alpha_{j},\beta_{j}\in\cohog{*}{M'}_{\ch}$. In particular, $c_{m}(\bR\pi_{12,*}\Hun)$ can be expressed as $\sum_{j}\alpha_{j}\otimes\beta_{j}$ for $\alpha_{j},\beta_{j}\in\cohog{*}{M'}_{\ch}$. By the above Claim, $cl(\Delta(M'))$ can also be expressed in the same form. In other words, we have $\cohog{*}{M'}_{\Delta}\subset\cohog{*}{M'}_{\ch}$.

Finally, it suffices to show that $\cohog{*}{M'}_{\Delta}=\cohog{*}{M'}_{\pure}$.  Let $\iota:M'\incl\barM'$ be an arbitrary smooth compactification of $M'$. By Deligne \cite[Corollaire 3.2.17]{Del H2}, the restriction map $\iota^{*}:\cohog{*}{\barM'}\to \cohog{*}{M'}_{\pure}$ is surjective. Since $\barM'$ is smooth and proper, the K\"unneth components of the diagonal class span $\cohog{*}{\barM'}$ (as a vector space). Since the class of $\Delta(\barM')$ restricts to that of $\Delta(M')$ under the map $\iota^{*}$, we have $\cohog{*}{M'}_{\Delta}=\cohog{*}{M'}_{\pure}$.   Together with the preceding discussion, we have $\cohog{*}{M'}_{\ch}=\cohog{*}{M'}_{\pure}$.
\end{proof}

\sss{The parabolic case of generation}
Now we consider an analogue of Theorem~\ref{th:Hit gen} for the  moduli space of parabolic Higgs bundles. 

Let $\Eun_{x_{0}}$ denotes the restriction of $\Eun$ to $\cM_{r,\D}(X,\cL)\times\{x_{0}\}$. Then the pullback of $\Eun_{x_{0}}$ to $\cM^{\pb}_{r,\D}(X,\cL)$ carries a universal filtration $F_{\bullet}\cE^{\univ}_{x_{0}}$. Therefore we get line bundles on $\cM^{\pb}_{r,\D}(X,\cL)$
\begin{equation*}
\cL^{\univ}_{i}=F_{i}\Eun_{x_{0}}/F_{i-1}\Eun_{x_{0}}, \quad i=1,2,\cdots, r.
\end{equation*}

Let $\cM'\subset \cM^{s}_{r,\D}(X,\cL)$ be an open substack such that the universal bundle $\Eun$ descends to a bundle $\un\cE^{\univ}$ on $M'\times X$, where $M'$ is the coarse moduli space of $\cM'$. Let $M'^{\pb}$ (resp. $\cM'^{\pb}$) be the preimage of $M'$ in $M^{\pb, s}_{r,\D}(X,\cL)$ (resp. $\cM'$ in $\cM^{\pb,s}_{r,\D}(X,\cL)$). Then $M'^{\pb}$ is a closed subscheme of the flag bundle associated to the bundle $\un\cE^{\univ}_{x_{0}}$, the restriction of $\un\cE^{\univ}$ to $M'\times\{x_{0}\}$. In particular, the (restrictions of the) line bundles $\cL^{\univ}_{i}$ on $\cM'^{\pb}$ descend to $M'^{\pb}$. We denote these descents still by $\un\cL^{\univ}_{i}$.

\begin{theorem}\label{th:glob par gen}
Suppose $\deg\cL\ge\deg\om_{X}+1$. Let $\cM'\subset \cM^{s}_{r,\D}(X,\cL)$ be an open substack such that the universal bundle $\Eun$ descends to a bundle $\un\cE^{\univ}$ on $M'\times X$, where $M'$ is the coarse moduli space of $\cM'$. Let $M'^{\pb}$ be the preimage of $M'$ in $M^{\pb, s}_{r,\D}(X,\cL)$. Then the pure part of the cohomology ring $\cohog{*}{M'^{\pb}}_{\pure}$ is generated by the pullbacks of the K\"unneth components of the universal Chern classes $c_{i}(\un{\cE}^{\univ})\in\cohog{2i}{M'\times X}$ ($1\leq i\leq r$) and the Chern classes $c_{1}(\un\cL^{\univ}_{i})\in \cohog{2}{M'^{\pb}}_{\pure}$ ($1\le i\le r$) as a $\QQ$-algebra.
\end{theorem}
\begin{proof}
The main idea of the proof is similar to that of Theorem~\ref{th:Hit gen}. We only sketch the modifications.  Instead of $\Hun$, we consider the complex $\cP^{\univ}$ on $M'^{\pb}\times M'^{\pb}\times X$ whose restriction to $\{\x\}\times\{\y\}\times X$ (where $\x=(\cE,\ph,F_{\bullet}\cE_{x_{0}}), \y=(\cF,\psi,F_{\bullet}\cF_{x_{0}})\in M'^{\pb}$ is
\begin{equation*}
\cP_{\x,\y}=[\un\Hom'(\cE,\cF)\xr{\delta(\ph,\psi)}\un\Hom'(\cE,\cF)\otimes\cL]
\end{equation*}
concentrated in degrees $0$ and $1$. Here $\un\Hom'(\cE,\cF)\subset \un\Hom(\cE,\cF)$ is the subsheaf of local maps $\cE\to\cF$ that preserve the flags at $x_{0}$. The map $\d(\ph,\psi)$ is defined by the similar formula as in the construction of $\Hun$. 

We claim that $\cohog{i}{X, \cP_{\x,\y}}=0$ for $i\ge2$ and any $\x=(\cE,\ph,F_{\bullet}\cE_{x_{0}}),\y=(\cF,\psi,F_{\bullet}\cF_{x_{0}})\in M'^{\pb}$. By Serre duality, this boils down to showing that any map $h:\cF\to \cE\otimes\cL^{-1}\otimes \om_{X}(x_{0})$  intertwining $\psi$ and $\ph$ and sending $F_{i}\cF_{x_{0}}$ to $F_{i-1}\cE_{x_{0}}$ must vanish. The argument is similar to that of Proposition~\ref{p:sm}(2).

Then we have the analogue of Claim~\ref{claim:diag} with the same proof: letting $m=\dim M'^{\pb}$, then $c_{m}(\bR\pi_{12}\cP^{\univ})\in \cohog{2m}{M'^{\pb}\times M'^{\pb}}$ is the class of a cycle supported on the diagonal $\Delta(M'^{\pb})$ with positive coefficient on each component.

The only difference in the remaining of the argument is the following. Let $\pi_{13},\pi_{23}:M'^{\pb}\times M'^{\pb}\times X\to M'^{\pb}\times X$ be the obvious projections.
\begin{claim} The Chern classes of $\bR\pi_{12,*}\cP^{\univ}$ can be expressed as a finite linear combination $\sum_{j}\a_{j}\otimes \b_{j}$, where each of $\a_{j}$ and $\b_{j}$ is a polynomial in the (pullbacks of) K\"unneth components of $c_{i}(\un{\cE}^{\univ})$ ($1\le i\le r$) and in $c_{1}(\un\cL^{\univ}_{i})$ ($1\le i\le r$).
\end{claim}
\begin{proof}[Proof of Claim] Let $\cK=\un\Hom(\pi_{13}^{*}\un{\cE}^{\univ}, \pi_{23}^{*}\un{\cE}^{\univ})$, a vector bundle on $M'^{\pb}\times M'^{\pb}\times X$.  Let $\cK'\subset \cK$ be the subsheaf preserving the flags of universal bundles at $x_{0}$. By the construction of $\cP^{\univ}$, it suffices to show that the total Chern class $c(\bR\pi_{12,*}\cK')$ can be expressed as a polynomial of the list of Chern classes. The quotient $\cK/\cK'$ is supported on $M'^{\pb}\times M'^{\pb}\times \{x_{0}\}$, and is a successive extension of $\un\cL^{\univ,-1}_{i}\boxtimes \un\cL^{\univ}_{j}$ for $1\le i< j\le r$. Therefore
\begin{equation}\label{KK'}
c(\bR\pi_{12,*}\cK')=c(\bR\pi_{12,*}\cK)\prod_{1\le i<j\le r}c(\un\cL^{\univ,-1}_{i}\boxtimes \un\cL^{\univ}_{j})^{-1}.
\end{equation}
On the other hand, $c(\cK)$ is a polynomial in  $\pi_{13}^{*}c_{i}(\un{\cE}^{\univ})$ and $\pi_{23}^{*}c_{i}(\un{\cE}^{\univ})$, hence by Grothendieck-Riemann-Roch, $c(\bR\pi_{12,*}\cK)\in\cohog{*}{M'^{\pb}\times M'^{\pb}}$ can be expressed as $\sum_{j}\a_{j}\otimes \b_{j}$ where $\a_{j}$ and $\b_{j}$ are polynomials in the K\"unneth components of $c_{i}(\un{\cE}^{\univ})$. By \eqref{KK'}, $c(\bR\pi_{12,*}\cK')$ has the desired form.
\end{proof}
The rest of the argument is the same as that of Theorem~\ref{th:Hit gen}: showing $\cohog{*}{M'^{\pb}}_{\Delta}\subset \cohog{*}{M'^{\pb}}_{\ch}$ and $\cohog{*}{M'^{\pb}}_{\Delta}=\cohog{*}{M'^{\pb}}_{\pure}$.
\end{proof}

\subsection{Generators for the cohomology rings of $J_{q/p}$ and $\Sp_{q/p}$}\label{sec:Surjectivity}
We shall apply the construction and results in Section~\ref{ss:glob gen} to a special case.

\sss{The curve $X$} Let $X$ be the Deligne-Mumford curve $\PP(p,1)=(\AA^{2}-\{(0,0)\})/\Gm$ where $\Gm$ acts by $t\cdot(\x,\y)=(t^{p}\x,t\y)$. Then $\infty_{X}:=[1,0]\in X$ is the unique point with nontrivial automorphism, and its automorphism group is $\mu_{p}$. 

The Picard group of $X$ is isomorphic to $\ZZ$, with ample generator denoted by $\calO_{X}(1/p)$ of degree $1/p$. For $q\geq0$, $\Gamma(X,\calO_{X}(q/p))$ is identified the $\CC[\x,\y]_{q}$, the degree $q$ part of the graded ring $\CC[\x,\y]$ (with $\deg\x=p, \deg\y=1$).

\sss{Hitchin moduli for $X$}\label{sss:Hit for X} Let $q\in\NN$ be coprime to $p$. Let $\cL=\cO_{X}(q/p)$ and $\D=\cO_{X}(-q(p-1)/2)$ (note $-q(p-1)/2\in \frac{1}{p}\ZZ$). As in Section~\ref{sss:Moduli Higgs}, we consider the moduli stack $\cM_{p,\D}(X,\cL)$ and its stable locus $\cM^{s}_{p,\D}(X,\cL)$. 

Also we fix the point $0_{X}:=[0,1]\in X$. By considering flags at $0_{X}$ preserved by Higgs fields we have the moduli stack $\cM^{\pb}_{p,\D}(X,\cL)$ of parabolic $\cL$-valued Higgs bundles, and the open substack $\cM^{\pb, s}_{p,\D}(X,\cL)$ as in Section~\ref{sss:par Hit}. Since $\om_{X}\cong\cO(-1-1/p)$, $\cL=\cO(q/p)$ satisfies $\deg\cL>\deg\om_{X}+1$. Therefore Proposition~\ref{p:sm}, Theorem~\ref{th:Hit gen} and \ref{th:glob par gen} are applicable.

Let $ \calA:=\prod_{i=1}^{p}\Gamma(X,\cL^{\otimes i})$ as an affine space over $\CC$. The $\CC$-points of $\calA$ is the vector space $\oplus_{i=1}^{p}\CC[\x,\y]_{iq}$.  Then we have the Hitchin fibration
\begin{equation*}
f:\cM_{p,\D}(X,\cL)\to\cA
\end{equation*}
which sends $(\calE,\ph)$ to the tuple $(\Tr(\wedge^{i}\ph))_{2\leq i\leq p}$, the coefficients of the characteristic polynomial of $\ph$.  

\sss{Torus actions} Let $\Grot$ be the one-dimensional torus which acts $X$ by $\Grot\ni t: [\x,\y]\mapsto[t^{-1}\x, \y]$. We fix the $\Grot$-equivariant structure on $\cL$ and $\D$ so that the actions on their fibers at $0_{X}$ are trivial. For $\cL$, this means that the induced action on $\Gamma(X,\cL)=\CC[\x,\y]_{q}$ is given by $\Grot\ni t: P(\x,\y)\mapsto P(t^{-1}\x,\y)$. The $\Grot$-action on $X$ and the $\Grot$-equivariant structures on $\cL$ and $\D$ induce a $\Grot$-action on $\cM_{r,\D}(X,\cL)$ and $\cM^{\pb}_{r,\D}(X,\cL)$. On the other hand,  let $\Gdil$ be one-dimensional torus which acts on $\cM_{r,\D}(X,\cL)$ and $M^{\pb}_{r,\D}(X,\cL)$ by scaling Higgs fields $\ph$. This action commutes with the action of $\Grot$. 

There is also a natural action of $\Grot\times\Gdil$ on $\cA$ making $f$ equivariant. 

Let $\Gm\incl\Grot\times\Gdil$ be the homomorphism given by $t\mapsto (t^{p},t^{q})$. Let $\Gm(q/p)$ be the image of this map, then $\Gm(q/p)$ acts on $\cM_{r,\D}(X,\cL)$ and $M^{\pb}_{r,\D}(X,\cL)$.  The torus $\Gm(q/p)$ also acts on $\calA$ by $\Gm(q/p)\ni t:P_{i}\mapsto t^{iq}P_{i}(t^{-p}\x,\y)$ for $f_{i}\in k[\x,\y]_{iq}$. The map $F$ is $\Gm(q/p)$-equivariant.

Let $\calA_{q/p}\subset\calA$ be fixed points under $\Gm(q/p)$, which is a copy of $\AA^{1}$ consisting of points $(0,\cdots,a\x^{q})$ for $a\in\AA^{1}$. Let
\begin{eqnarray*}
\cM_{1}=f^{-1}(0,\cdots,0, -\x^{q}).
\end{eqnarray*}
We define $\cM^{\pb}_{1}$ to be the preimages of  $\cM_{1}$ in $\cM^{\pb,s}_{r,\D}(X,\cL)$.

\sss{Elliptic locus}\label{sss:ell} For each $a\in \cA$ one can define the spectral $Y_{a}$ in the total space of $\cL$ over $X$. The projection $\th_{a}:Y_{a}\to X$ is a finite morphism of degree $p$. Let $\cA'\subset \cA$ be the open subset where the spectral curve is irreducible and reduced. Let $\cY'\to \cA'$ be the universal family of (integral) spectral curves over $\cA'$. Then $\cA'$ is invariant under the $\Grot\times\Gdil$-action. Let $\cM'\subset \cM_{p,\D}(X,\cL)$ and $\cM'^{\pb}\subset \cM^{\pb}_{p,\D}(X,\cL)$ be the preimages of $\cA'$. 

\begin{lemma}\label{l:ell}
\begin{enumerate}
\item\label{M' stable} We have $\cM'\subset \cM^{s}_{p,\D}(X,\cL)$, and $\cM'^{\pb}\subset \cM^{\pb,s}_{p,\D}(X,\cL)$.
\item\label{cJac} Let $M'$ be the coarse moduli space of $\cM'$. Then $M'$ is isomorphic to the relative compactified Jacobian $\ov\Jac(\cY/\cA')$ for the family of spectral curves $\cY\to \cA'$.
\item\label{f proper} The map $f':\cM'\to \cA'$ (restriction of $f$ to $\cA'$) is proper.
\end{enumerate}
\end{lemma}
\begin{proof}
\eqref{M' stable} If $(\cE,\ph)\in \cM_{p,\D}(X,\cL)$ has image $a=f(\cE,\ph)\in \cA'$, $(\cE,\ph)$ does not admit any nontrivial Higgs subbundle (for otherwise $Y_{a}$ would contain a proper subscheme which is also a curve), therefore $f^{-1}(\cA')\subset \cM^{s}_{p,\D}(X,\cL)$. 

\eqref{cJac} Let $\cPic(\cY/\cA')_{\D}$ be the moduli stack of triples $(a,\cF,\tau)$ where $a\in \cA'$, $\cF$ is a torsion-free rank one sheaves on $Y$ and $\tau$ is an isomorphism $\det(\th_{a*}\cF)\cong \D$. For $(a,\cF,\tau)\in\cPic(\cY/\cA')_{\D}$, $\th_{a*}\cF$ is a vector bundle over $X$ of rank $p$ together with a Higgs field $\ph: \th_{a*}\cF\to \th_{a*}\cF\otimes\cL$ coming from the fact that $\cL^{-1}\subset \th_{a*}\calO_{Y_{a}}$ which acts on $\th_{a*}\cF$. Therefore $(a,\cF,\tau)\mapsto (\th_{a*}\cF,\tau)$ gives an isomorphism $\cPic(\cY/\cA')_{\D}\isom \cM'$. 

Note that for $a\in \cA'$, $\th_{a*}\cO_{Y_{a}}\cong \cO_{X}\oplus \cL^{-1}\oplus\cdots\oplus \cL^{-(p-1)}$, we have $\deg(\th_{a*}\cO_{Y_{a}})=-q(p-1)/2=\deg\D$. Therefore for any $(a, \cF,\tau)\in \cPic(\cY/\cA')_{\D}$, $\cF$ has the same Euler characteristic as  $\cO_{Y_{a}}$. Forgetting $\tau$ and passing to the isomorphism classes of $\cF$ gives a map $\om: \cPic(\cY/\cA')_{\D}\isom \ov\Jac(\cY/\cA')$. The map $\om$ is a $\mu_{p}$-gerbe, realizing $\ov\Jac(\cY/\cA')$ as the coarse moduli space of $\cPic(\cY/\cA')_{\D}\cong \cM'$.

\eqref{f proper} Since $\ov\Jac(\cY/\cA')\to \cA'$ is proper and $\cM'\to\ov\Jac(\cY/\cA')$ is a $\mu_{p}$-gerbe (hence proper), $f': \cM'\to \cA'$ is also proper.
\end{proof}

\begin{lemma}\label{l:Sp Hit} 
\begin{enumerate}
\item We have $\cA_{q/p}-\{0\}\subset \cA'$.
\item Let $M_{1}$ be the coarse moduli space of $\cM_{1}$. Then there is a canonical isomorphisms $J_{q/p}\cong M_{1}$ which is equivariant with respect to the $\Gm(q/p)$-actions.
\item Let $M^{\pb}_{1}$ be the coarse moduli space of $\cM^{\pb}_{1}$. Then there is a canonical homeomorphism $\Sp_{q/p}\to M^{\pb}_{1}$ which is equivariant with respect to the $\Gm(q/p)$-actions.
\end{enumerate}
\end{lemma}
\begin{proof}(1) We claim that the spectral curve $Y_{a}$ over $a\in \cA_{q/p}-\{0\}$ is isomorphic to $C_{q/p}$ in a $\Gm(q/p)$-equivariant way. For then $Y_{a}$ is integral hence $\cA_{q/p}-\{0\}\subset \cA'$. Part (2) and (3) follow from Lemma~\ref{l:ell}\eqref{cJac}. By $\Gdil$-symmetry, it suffices to treat the case $a=(0,\cdots, 0,-\x^{q})$.  Using the affine coordinate $x=\x/\y^{p}$ for the open subscheme $\AA^{1}=X-\{\infty_{X}\}\subset X$, the spectral curve $Y_{a}$ restricted to $X-\{\infty\}$ is the plane curve $y^{p}-x^{q}=0$. The restriction of $\th_{a}:Y_{a}\to X$ over $\infty_{X}$ is $\pt\to [\pt/\mu_{p}]$, therefore  $\th_{a}: Y_{a}\to X$ is unramified over $\infty_{X}$. In particular, $Y_{a}$ is a compactification of $y^{p}-x^{q}$ which is smooth at infinity, hence $Y_{a}\cong C_{q/p}$. The point $a\in\cA_{q/p}$ being fixed by $\Gm(q/p)$, $Y_{a}$ carries a $\Gm(q/p)$-action. It is easy to see that the isomorphism $Y_{a}\cong C_{q/p}$ is $\Gm(q/p)$-equivariant. 

(2) Since we have a $\Gm(q/p)$-equivariant isomorphism $Y_{1}\cong C_{q/p}$, part (2) follows from Lemma~\ref{l:ell}\eqref{cJac}. 

(3) From (2) and \eqref{SGJ} we get a canonical homeomorphism $\SG\to M_{1}$. From the construction in Lemma~\ref{l:ell}\eqref{cJac} it is easy to see that this homeomorphism lifts to a homeomorphism $\Sp_{q/p}\to M^{\pb}_{1}$.
\end{proof}

\sss{Contractions}\label{sss:contraction}
The action of $\Gm(q/p)$ contracts $\cA'$ to $\cA_{q/p}-\{0\}$ by the map
\begin{eqnarray}\label{tau}
\tau: \cA'&\to& \cA_{q/p}-\{0\}\cong \AA^{1}-\{0\}\\
(P_{2},\cdots, P_{p})&\mapsto& \mbox{the coefficient of $\xi^{q}$ in $P_{p}$.}
\end{eqnarray} 

By Lemma~\ref{l:ell}\eqref{f proper}, $\cM'$ is proper over $\cA'$, therefore the action of $\Gm(q/p)$ contracts it to $\cM|_{\cA_{q/p}-\{0\}}$ by a map which is a homotopy equivalence
\begin{equation}\label{wt tau}
\wt\tau: \cM'\to \cM|_{\cA_{q/p}-\{0\}}.
\end{equation}
On the other hand, let $(u,v)\in\ZZ^{2}$ be such that $up-vq=1$. Then the action of $\Gm(u/v)$ (the subtorus given by $(t^{u}, t^{v})\in \Grot\times\Gdil$) on $\cA_{q/p}-\{0\}$ is simply transitive. Therefore the action map gives an isomorphism
\begin{equation*}
\Gm(u/v)\times\cM_{1}\isom \cM|_{\cA_{q/p}-\{0\}}.
\end{equation*}
Similarly we have the contraction
\begin{equation*}
\cM'^{\pb}\xr{\wt\tau^{\pb}}\cM^{\pb}|_{\cA_{q/p}-\{0\}}\cong \Gm(u/v)\times\cM^{\pb}_{1}.
\end{equation*}
Let $M',M'^{\pb}, M_{1}, M^{\pb}_{1}$ be the coarse moduli spaces of $\cM',\cM'^{\pb}, \cM_{1}, \cM^{\pb}_{1}$.

\begin{cor}\label{c:contract to fiber} The maps $\wt\tau$ and $\wt\tau^{\pb}$ and the isomorphisms in Lemma~\ref{l:Sp Hit} induce isomorphisms of cohomology rings
\begin{eqnarray*}
\cohog{*}{\Gm}\otimes \cohog{*}{J_{q/p}}\cong \cohog{*}{\Gm(u/v)}\otimes \cohog{*}{M_{1}}\isom \cohog{*}{M'},\\
\cohog{*}{\Gm}\otimes \cohog{*}{\Sp_{q/p}}\cong \cohog{*}{\Gm(u/v)}\otimes \cohog{*}{M^{\pb}_{1}}\isom \cohog{*}{M'^{\pb}}.
\end{eqnarray*}
In particular, the restriction maps induce isomorphisms
\begin{eqnarray*}
\cohog{*}{M'}_{\pure}\isom \cohog{*}{J_{q/p}},\\
\cohog{*}{M'^{\pb}}_{\pure}\isom \cohog{*}{\Sp_{q/p}},
\end{eqnarray*}
\end{cor}

\begin{lemma}\label{l:Eun} The universal bundle $\Eun$ descends to $M'\times X$.
\end{lemma}
\begin{proof} The $\mu_{p}$-gerbe $\cM'\to M'$ gives a class $c\in \upH^{2}_{\et}(M',\mu_{p})$. We denote its pullback to $M'\times X$ by $c_{X}$. The obstruction to descending $\Eun$ to $M'\times X$ lies in the cohomological Brauer group $\upH^{2}_{\et}(M'\times X,\Gm)_{\tor}$, and is the image of $c_{X}$ under the natural map 
\begin{equation*}
\upH^{2}_{\et}(M'\times X,\mu_{p})\to \upH^{2}_{\et}(M'\times X,\Gm)[p].
\end{equation*}
The homotopy retract $\wt\tau$ in \eqref{wt tau} induces a homotopy retract
\begin{equation*}
\rho: M'\to M|_{\cA_{q/p}-\{0\}}\cong \Gm(u/v)\times M_{1}
\end{equation*}
by passing to coarse moduli spaces. Hence $\rho$ induces an isomorphism on \'etale cohomology groups, and gives a commutative diagram
\begin{equation*}
\xymatrix{ \upH^{2}_{\et}(\Gm(u/v)\times M_{1}\times X, \mu_{p})\ar[rr]^-{(\rho\times\id_{X})^{*}}_{\sim}\ar[d] &&    \upH^{2}_{\et}(M'\times X, \mu_{p})\ar[d]      \\
 \upH^{2}_{\et}(\Gm(u/v)\times M_{1}\times X, \Gm)[p]\ar[rr]^-{(\rho\times\id_{X})^{*}} &&  \upH^{2}_{\et}(M'\times X, \Gm)[p]}
 \end{equation*}
Therefore, it suffices to show that the image of $c_{X}|_{\Gm(u/v)\times M_{1}\times X}$ in $\upH^{*}_{\et}(\Gm(u/v)\times M_{1}\times X, \Gm)$ vanishes, i.e., $\Eun|_{\Gm(u/v)\times \cM_{1}\times X}$ descends to $\Gm(u/v)\times M_{1}\times X$. Since $\Eun$ is $\Gm(u/v)$-equivariant, it suffices to show that $\Eun|_{\cM_{1}\times X}$ descends to $M_{1}\times X$. Recall the isomorphism $M_{1}\cong J_{q/p}$ in Lemma~\ref{l:Sp Hit}. Over $J_{q/p}\times C_{q/p}$ we have a universal rank one torsion-free sheaf $\cF^{\univ}$  (for example we may identify $J_{q/p}$ with the fine moduli of rank one torsion-free sheaves on $C_{q/p}$ with a trivialization at the smooth point $\infty_{C}\in C_{q/p}$). Let $\th_{1}: C_{q/p}\to X$ be the projection when viewing $C_{q/p}$ as the spectral curve of $(0,\cdots,0,-\x^{q})\in \cA_{q/p}$. Then $(\id\times\th_{1})_{*}\cF^{\univ}$ is a vector bundle of rank $p$ over $J_{q/p}\times X$, which can be taken to be the descent of $\Eun|_{\cM_{1}\times X}$ after identifying $J_{q/p}$ with $M_{1}$. This completes the proof.
\end{proof}

\subsubsection{Proof of Theorem~\ref{th:gen}} We apply Theorem~\ref{th:Hit gen} to the open substack $\cM'$ defined in Section~\ref{sss:ell}, which satisfies the assumption of Theorem~\ref{th:Hit gen} by Lemma~\ref{l:Eun}. We denote the descent of $\Eun$ to $M'\times X$ by $\un{\cE}^{\univ}$, and let
\begin{equation*}
\un{\cE}^{\univ}_{1}:=\un{\cE}^{\univ}|_{M_{1}\times X}.
\end{equation*}
Theorem~\ref{th:Hit gen} shows that $\cohog{*}{M'}_{\pure}$ is generated by the K\"unneth components of $c_{i}(\Eun), 1\le i\le p$. By Corollary~\ref{c:contract to fiber}, $\cohog{*}{M_{1}}$ is generated by the K\"unneth components of the restrictions $c_{i}(\un{\cE}^{\univ})|_{M_{1}\times X}$. We may write according to the K\"unneth decomposition of $\cohog{*}{M_{1}\times X}$:
\begin{equation*}
c_{i}(\un{\cE}^{\univ})|_{M_{1}\times X}=\a_{i}\otimes 1+\b_{i}\otimes[X].
\end{equation*}
where $\a_{i}\in \cohog{2i}{M_{1}}$ and $\b_{i}\in \cohog{2i-2}{M_{1}}$ for $1\leq i\leq p$.

\begin{claim} For $1\leq i\leq p$, we have $\a_{i}=0$.
\end{claim}
\begin{proof}[Proof of Claim]
In fact,  $\a_{i}$ is the restriction of $c_{i}(\un{\cE}^{\univ})$ to $M_{1}\times\{z\}$ for any $z\in X$. In Section~\ref{Xqp}, we introduced $\wt J_{q/p}$ replacing the trivialization at $\infty_{C}$ by a trivialization over $C_{q/p}-\{0_{C}\}$. The pullback of $\un{\cE}^{\univ}$ to $\wt J_{q/p}\times (C_{q/p}-\{0_{C}\})$ is canonically trivialized. Therefore, the pullback of $\un{\cE}^{\univ}$ to $\wt J_{q/p}\times\{z\}$ is also trivialized for any $z\in X-\{0_{X}\}$, hence the pullback of $c_{i}(\un{\cE}^{\univ})$ to $\wt J_{q/p}\times \{z\}$ vanishes. Since $\wt J_{q/p}\to J_{q/p}$ is a homeomorphism, the restriction of $c_{i}(\un{\cE}^{\univ})$ to $M_{1}\times\{z\}\cong J_{q/p}\times\{z\}$ also vanishes. Hence $\a_{i}=0$.
\end{proof}

\begin{claim} Under the isomorphism $J_{q/p}\cong M_{1}$ in Lemma~\ref{l:Sp Hit}, $e_{i}$ corresponds to $\b_{i}$ for $2\le i\le p$.
\end{claim}
\begin{proof}[Proof of Claim] Recall $\pi_{x}: C_{q/p}\to \PP^{1}$ is the map given by $(x,y)\mapsto x$. Then $x$ factors as $C_{q/p}\xr{\th_{1}} X\xr{\g}\PP^{1}$, where the second map $\g$ identifies $\PP^{1}$ with the coarse moduli space of $X=\PP(p,1)$. Let $\cF^{\univ}$ be the universal sheaf on $J_{q/p}\times C_{q/p}$, and recall $\Eun_{x}=(\id\times \pi_{x})_{*}\cF^{\univ}$ is a vector bundle of rank $p$ on $J_{q/p}\times \PP^{1}$. On the other hand, $\un{\cE}^{\univ}_{1}:=\un{\cE}^{\univ}|_{M_{1}\times X}$ is constructed as $(\id\times \th_{1})_{*}\cF^{\univ}$. Therefore $\Eun_{x}\cong (\id\times\g)_{*}(\un{\cE}^{\univ}_{1})$. In particular, we have a natural inclusion $(\id\times\g)^{*}\Eun_{x}=(\id\times\g)^{*}(\id\times\g)_{*}\un{\cE}^{\univ}_{1}\to \un{\cE}^{\univ}_{1}$ whose cokernel $\cK$ is a vector bundle concentrated along $J_{q/p}\times \{\infty_{X}\}$. Moreover, pulling back $\cK$ to $\wt J_{q/p}\times\{\infty_{X}\}$ (see Section~\ref{sss:relXJ} for the definition of $\wt J_{q/p}$), $\cK$ is a trivial bundle because both $(\id\times\g)^{*}\Eun_{x}$ and $\un{\cE}^{\univ}_{1}$ are canonically trivialized over $\wt J_{q/p}\times (X-\{0_{X}\})$. Therefore $\ch(\cK)$ is a multiple of $1\otimes[X]\in \cohog{2}{M_{1}\times X}$, and we conclude that
\begin{equation}\label{diff Eun}
\ch(\un{\cE}^{\univ}_{1})-(\id\times\g)^{*}\ch(\Eun_{x})=\ch(\cK) \in \QQ\cdot 1\otimes [X].
\end{equation}
We know from the previous Claim that that $c_{i}(\un{\cE}^{\univ}_{1})=\b_{i}\otimes[X]$, therefore the product of more than one $c_{i}(\un{\cE}^{\univ}_{1})$ is zero, hence
\begin{equation}\label{ch E}
\ch(\un{\cE}^{\univ}_{1})=p+\sum_{i=1}^{p}(-1)^{i-1}c_{i}(\un{\cE}^{\univ}_{1})/(i-1)!.
\end{equation}
Same argument as in the previous Claim shows that $c_{i}(\Eun_{x})=e_{i}\otimes[X]$, and hence
\begin{equation}\label{ch Ep}
\ch(\Eun_{x})=p+\sum_{i=1}^{p}(-1)^{i-1}c_{i}(\Eun_{x})/(i-1)!.
\end{equation}
Combining \eqref{ch E},\eqref{ch Ep} and using \eqref{diff Eun}, we see that $c_{i}(\un{\cE}^{\univ}_{1})=(\id\times\g)^{*}c_{i}(\Eun_{x})$ for $2\leq i\leq p$, which then implies $\b_{i}=e_{i}$.
\end{proof}

We finish the proof of Theorem~\ref{th:gen}. Since the classes $\alpha_{1},\cdots, \alpha_{p},\beta_{2},\cdots,\beta_{p}$ generate $\cohog{*}{M_{1}}$, and $\alpha_{i}=0$,  $\beta_{2},\cdots,\beta_{p}$ then generate $\cohog{*}{M_{1}}$. Hence, under the isomorphism $J_{q/p}\cong M_{1}$, $e_{2},\cdots, e_{p}$ generate $\cohog{*}{J_{q/p}}$. \qed

\sss{Proof of Theorem~\ref{th:par gen}}
We apply Theorem~\ref{th:glob par gen} to the open substack $\cM'$ defined in Section~\ref{sss:ell}, which satisfies the assumption of Theorem~\ref{th:Hit gen} by Lemma~\ref{l:Eun}.  Theorem~\ref{th:glob par gen} shows that $\cohog{*}{M'^{\pb}}_{\pure}$ is generated by the K\"unneth components of $c_{i}(\Eun), 1\le i\le p$ and $c_{1}(\cL^{\univ}_{i}), 1\le i\le p$. By Corollary~\ref{c:contract to fiber} and Lemma~\ref{l:Sp Hit}, $\cohog{*}{M^{\pb}_{1}}\cong \cohog{*}{\Sp_{q/p}}$ is generated by the restrictions of these cohomology classes to $M^{\pb}_{1}$. Note under the homeomorphism $\Sp_{q/p}\to M^{\pb}_{1}$, the line bundles $\cL^{\univ}_{i}$ pullback to $\cL_{i}$ on $\Sp_{q/p}$, $1\le i\le p$.

\subsection{Equivariant version} 
The vector bundle $\cE^{\univ}_{x}$ on $J_{q/p}\times\PP^{1}$ (see Section~\ref{sss:Ex}) is $\Gm(q/p)$-equivariant, hence the equivariant Chern classes $c^{\Gm}_{i}(\Eun_{x})\in \eqcoh{2i}{J_{q/p}\times \PP^{1}}$ are defined, and reduce to $c_{i}(\Eun_{x})$ in $\cohog{2i}{J_{q/p}\times \PP^{1}}$. Restricting to $J_{q/p}\times\{0\}$ we get classes $c^{\Gm}_{i}(\Eun_{x})|_{J_{q/p}\times 0}\in \eqcoh{2i}{J_{q/p}}$.

\begin{lemma}\label{l:equiv gen}
\begin{enumerate}
\item For $1\leq i\leq p$, the classes $c^{\Gm}_{i}(\Eun_{x})|_{J_{q/p}\times \{0\}}\in \eqcoh{2i}{J_{q/p}}$ are divisible by $\ep$ (the equivariant parameter). Define
\begin{equation*}
\wt e_{i}=p^{-1}\ep^{-1}c^{\Gm}_{i}(\Eun_{x})|_{J_{q/p}\times \{0\}}\in \eqcoh{2i-2}{J_{q/p}}.
\end{equation*}
\item For $2\leq i\leq p$, the image of $\wt e_{i}$ in $\cohog{2i-2}{J_{q/p}}$ is $e_{i}$.
\end{enumerate}
\end{lemma}
\begin{proof}
The one-dimensional torus $\Gm(q/p)$ acts on $\PP^{1}$ via the $p$-th power of the usual rotation action. Localizing $\PP^{1}$ to the two fixed points $0$ and $\infty$ we get an isomorphism
\begin{equation*}
(\iota^{*}_{0},\iota^{*}_{\infty}): \eqcoh{*}{\PP^{1}}\isom \QQ[\ep]\times_{\QQ}\QQ[\ep]
\end{equation*}
where the right side consists of pairs $(h_{0}, h_{\infty})\in \QQ[\ep]^{2}$ with the same constant term. The map to the usual cohomology $\eqcoh{*}{\PP^{1}}\to \cohog{*}{\PP^{1}}=\QQ\cdot 1\oplus \QQ\cdot[\PP^{1}]$ takes $(h_{0},h_{\infty})$ to $(h_{0}(0), \frac{h_{0}-h_{\infty}}{p\ep}|_{\ep=0})$.

Taking product with $J_{q/p}$ we get a similar isomorphism given by restriction to $J_{q/p}\times\{0\}$ and to $J_{q/p}\times\{\infty\}$
\begin{equation}\label{i0}
(i^{*}_{0},i^{*}_{\infty}): \eqcoh{*}{J_{q/p}\times \PP^{1}}\isom \eqcoh{*}{J_{q/p}}\times_{\cohog{*}{J_{q/p}}}\eqcoh{*}{J_{q/p}}.
\end{equation}
The map to the usual cohomology $\eqcoh{*}{J_{q/p}\times \PP^{1}}\to \cohog{*}{J_{q/p}\times\PP^{1}}=\cohog{*}{J_{q/p}}\otimes 1\oplus \cohog{*}{J_{q/p}}\otimes[\PP^{1}]$ takes the form 
\begin{equation}\label{eqcoh to cohog}
(h_{0},h_{\infty})\mapsto h_{0}|_{\ep=0}\otimes1+ \frac{h_{0}-h_{\infty}}{p\ep}|_{\ep=0}\otimes [\PP^{1}]
\end{equation}
where $(-)|_{\ep=0}$ is the map $\eqcoh{*}{J_{q/p}}\to \cohog{*}{J_{q/p}}$. 

Now we apply the map \eqref{i0} to $c^{\Gm}_{i}(\Eun_{x})$. We can replace $J_{q/p}$ by the homeomorphic $\wt J_{q/p}$ (see Section~\ref{sss:relXJ}), and conclude that the restriction of $\Eun_{x}$ to $\wt J_{q/p}\times\{\infty\}$ is a direct sum of trivial bundles $\cO,\cO(1),\cdots, \cO(p-1)$ with the $\Gm$-action indicated by the twists. Therefore, $i^{*}_{\infty}c^{\Gm}_{i}(\Eun_{x})$ is the $i$th elementary symmetric polynomial of the set $\{0,\ep,\cdots,(p-1)\ep\}$, hence a multiple of $\ep^{i}$ viewed as an element in  $\eqcoh{2i}{J_{q/p}}$. Since $i^{*}_{0}c^{\Gm}_{i}(\Eun_{x})$ is congruent to $i^{*}_{\infty}c^{\Gm}_{i}(\Eun_{x})$ mod $\ep$, we see that $i^{*}_{0}c^{\Gm}_{i}(\Eun_{x})$ is divisible by $\ep$ for $2\leq i\leq p$. This proves (1). 

By \eqref{eqcoh to cohog}, $e_{i}$ is the reduction mod of $\ep$ of $p^{-1}\ep^{-1}(i^{*}_{0}c^{\Gm}_{i}(\Eun_{x})-i^{*}_{\infty}c^{\Gm}_{i}(\Eun_{x}))$. Since $i^{*}_{\infty}c^{\Gm}_{i}(\Eun_{x})$ is a multiple of $\ep^{i}$, we see that $e_{i}$ is the reduction mod $\ep$ of  $\wt e_{i}=p^{-1}\ep^{-1}i^{*}_{0}c^{\Gm}_{i}(\Eun_{x})$ for $2\leq i\leq p$. This proves (2).
\end{proof}

\begin{cor}\label{c:equiv gen} The equivariant cohomology ring $\eqcoh{*}{J_{q/p}}$ is generated by the equivariant parameter $\ep\in\eqcoh{2}{\pt}$ together with $\wt e_{i}, 2\leq i\leq p$.
\end{cor}
\begin{proof} Since $J_{q/p}$ is equivariantly formal (see Section~\ref{sss:paving}), therefore $\eqcoh{*}{J_{q/p}}\otimes_{\QQ[\ep]}\QQ\cong \cohog{*}{J_{q/p}}$ (here $\QQ$ is viewed as the $\QQ[\ep]$-module $\QQ[\ep]/(\ep)$). By Lemma~\ref{l:equiv gen}(2), $\wt e_{i}$ are liftings of $e_{i}$ for $2\leq i\leq p$. The generation statement for $\eqcoh{*}{J_{q/p}}$ then follows from Theorem~\ref{th:gen} by the graded version of the Nakayama's lemma.
\end{proof} 

Same argument shows the parabolic version.

\begin{cor} The equivariant cohomology ring $\eqcoh{*}{\Sp_{q/p}}$ is generated by the equivariant parameter $\ep\in\eqcoh{2}{\pt}$, the pullbacks of $\wt e_{i}$ for $2\leq i\leq p$ and $c_{1}^{\Gm}(\cL_{i}), 1\le i\le p$.
\end{cor}

\subsection{Smooth model for $J_{q/p}$}\label{ssec:WeightFilt}
The curve $C_{q/p}$ naturally sits inside the weighted projective plane $\PP(1,p,q)$ with equation $F_{0}(z_{0},z_{1},z_{2})=z_{1}^{q}-z_{2}^{p}=0$. Let $\cB$ be the space of homogeneous degree $pq$ polynomials $F(z_{0},z_{1},z_{2})$ in the weighted variables $(z_{0},z_{1},z_{2})$ (with degrees $1,p,q$ respectively) of the form\begin{equation*}
F(z_{0},z_{1}z_{2})=\sum_{i,j\ge0, ip+jq\le pq}a_{ij}z_{1}^{i}z_{2}^{j}z_{0}^{pq-ip-jq}
\end{equation*}
with
\begin{equation*}
a_{0p}=1, \quad a_{q0}=-1.
\end{equation*}
The $(a_{ij})$ for $i,j\ge0, ip+jq\le pq, (i,j)\ne (0,p)$ or $(q,0)$ give affine coordinates on $\cB$ and hence an isomorphism $\cB\cong \AA^{(p+1)(q+1)/2-1}$. Let $\cB'\subset \cB$ be the open subset where $F$ defines an integral curve. Then we have a universal family of curves
\begin{equation*}
\pi: \cC\to \cB'
\end{equation*}
defined by $F(z_{0},z_{1},z_{2})=0$ inside $\PP(1,p,q)$ for $F\in \cB'$. Note all these curves are ordinary schemes (they don't pass through the points in $\PP(1,p,q)$ with nontrivial automorphisms). Moreover all these curves pass through the point $[0,1,1]$, so that $\pi$ has a section. Therefore the relative compactified Jacobian $\cJ(\cC/\cB')$ is defined.

The one-dimensional torus $\Gm(q/p)$ acts on $\PP(1,p,q)$ by $t\cdot [z_{0},z_{1},z_{2}]\mapsto [tz_{0}, z_{1},z_{2}]$. This action contracts $\cB$ and $\cB'$ to the point $F_{0}=z_{1}^{q}-z_{2}^{p}$. 

\begin{prop} The relative compactified Jacobian $\cJ(\cC/\cB')$ is smooth and it contracts to the central fiber $J_{q/p}$ under the action of $\Gm(q/p)$.
\end{prop}
\begin{proof}
The moduli space $\cJ(\cC/\cB')$ is closely related to a Hitchin moduli space for $X=\PP(p,1)$. More precisely, consider the moduli stack $\cM_{p,\D}(X,\cL)$ for $\D=\cO_{X}(-(p-1)q/2)$ and $\cL=\cO_{X}(q/p)$ as in Section~\ref{sss:Hit for X}. The Hitchin base $\cA$ classifies polynomials $y^{p}+P_{1}(\x,\y)y^{p-1}+\cdots +P_{p}(\x,\y)$ in variables $\x,\y$ and $y$ such that $P_{i}(\x,\y)\in \CC[\x,\y]_{iq}$. We have an embedding
\begin{eqnarray*}
\cB &\incl& \cA\\
F(z_{0},z_{1},z_{2}) &\mapsto & F(\y,\x,y)
\end{eqnarray*}
which identifies $\cB$ as a hyperplane in $\cA$ defined by $P_{p}(\x,\y)=-\x^{q}+\cdots$. For $F\in \cB$ viewed as a point in $\cA$, the corresponding spectral curve is naturally isomorphic to the curve $\cC_{F}\subset \PP(1,p,q)$ defined by $F=0$. Lemma~\ref{l:ell} implies there is a natural map $\cM_{r,\D}(X,\cL)|_{\cB'}\to \cJ(\cC/\cB')$ which is a $\mu_{p}$-gerbe. On the other hand, the action of $\Gm(u/v)$ (see Section~\ref{sss:contraction}) identifies $\cA'$ with $\Gm(u/v)\times \cB'$, hence $\cM_{r,\D}(X,\cL)|_{\cA'}\cong\Gm(u/v)\times \cM_{r,\D}(X,\cL)|_{\cB'}$. By Proposition~\ref{p:sm}, $\cM_{r,\D}(X,\cL)$ is smooth, hence so are $\cM_{r,d}(X,\cL)|_{\cB'}$ and $\cJ(\cC/\cB')$.

Since the $\Gm(q/p)$-action contracts $\cB'$ to $F_{0}$, and $\cJ(\cC/\cB')$ is proper over $\cB'$, we conclude that the same action also contracts $\cJ(\cC/\cB')$ to its central fiber $J_{q/p}$.
\end{proof}

Consider the map
\begin{eqnarray*}
\pi_{\cB,\infty}:\cB'&\to&\AA^{p+q-1}\\
F(z_{0},z_{1},z_{2}) &\mapsto& (a_{00},a_{10},\cdots, a_{q-1,0}, a_{01},\cdots,a_{0,p-1}).
\end{eqnarray*}

\begin{conj}
\begin{enumerate}
\item There is a canonical Poisson structure on $\cJ(\cC/\cB')$ such that the projection $h: \cJ(\cC/\cB')\to \cB'$ is an algebraically completely integrable system.
\item The generic fibers of the map
\begin{equation*}
\pi_{\infty}: \cJ(\cC/\cB')\to\cB'\xr{\pi_{\cB,\infty}}\AA^{p+q-1}
\end{equation*}
are smooth and are the symplectic leaves of the Poisson structure on $\cJ(\cC/\cB')$.
\end{enumerate}
\end{conj}

The result \cite[Proposition 3.1]{BH} is very close to the statement of the conjecture.  We expect that a proof of the conjecture could be obtained by the appropriate modification of the argument from \cite{BH}.



\section{The cohomology rings of $J_{q/p}$ and $\Sp_{q/p}$: relations}\label{sec:Relations}

We preserve the notation from Section~\ref{ss:Cpq} and Section~\ref{sec:Surjectivity}. In this section we first give a presentation of the cohomology ring $\eqcoh{*}{J_{q/p}}$. Then we complete the proof of Theorem~\ref{th:main}.

\subsection{A two-parameter family of rings} 

\sss{Definition of $R$}\label{sss:R} Let $\ep,s, e_{2},\cdots, e_{p}, f_{2},\cdots, f_{q}$ and $z$ be indeterminates. Let
\begin{equation}\label{e1f1}
e_{1}=\frac{q(p-1)}{2}s; \quad f_{1}=\frac{p(q-1)}{2}s.
\end{equation}
and
$$A(z)=z^{p}+\sum_{i=1}^{p}p\ep e_{i}z^{p-i}; \quad B(z)=z^{q}+\sum_{j=1}^{q}q\ep f_{j}z^{q-j}.$$ 
We view both $A$ and $B$ as polynomials in $z$ with coefficients in $\QQ[\ep, e_{2},\cdots, e_{p},f_{2},\cdots, f_{q}]$. Let $F_{d}\in \QQ[\ep, s, e_{2},\cdots, e_{p},f_{2},\cdots, f_{q}]$ ($0\leq d\leq pq$) be the coefficient of $z^{d}$ in
\begin{equation}\label{AB}
\prod_{j=0}^{q-1}A(z+jps\ep)-\prod_{i=0}^{p-1}B(z+iqs\ep).
\end{equation}
Let
$$\wt R=\QQ[\ep, s,e_2,\cdots,e_{p},f_{2},\cdots,f_q]/(F_{0},F_{1},\cdots, F_{pq})$$ 
Finally, let $\wt R(\ep)$ be the $\ep$-power torsion elements of $\wt R$, and let
\begin{equation*}
R:=\wt R/\wt R(\ep)
\end{equation*}
be the $\ep$-torsion-free quotient of $\wt R$.

\subsubsection{Bigrading}
We give a bigrading on $\QQ[\ep,s,e_{2},\cdots, e_{p}, f_{2},\cdots,f_{q}]$ in which the generators get degrees
\begin{equation*}
\deg(\ep)=(1,0); \quad \deg(s)=(0,1); \quad \deg(e_{i})=(i-1,i); \quad \deg(f_{i})=(i-1,i).
\end{equation*}
It is easy to see that the generators  $F_{d}$ ($0\leq d\leq pq$) of the ideal we quotient by to get $\wt R$ are homogeneous under the bigrading, hence $\wt R$ inherits a bigrading, and so does $R$.

Let $R=\oplus_{a,b\in\ZZ_{\geq0}}(_{a}R_{b})$ be the decomposition of $R$ into graded pieces under the bigrading. Then $\ep: {}_{a}R_{b}\to {}_{a+1}R_{b}$ and $s: {}_{a}R_{b}\to {}_{a}R_{b+1}$. 

From Lemma~\ref{l:Oqp}(2), we immediately see
\begin{lemma} There is a ring isomorphism
\begin{equation}\label{RO}
R_{\ep=1,s=0}\cong\cO_{q/p}
\end{equation}
sending $e_{i}\in R_{\ep=1,s=0}$ to the same-named $e_{i}\in \cO_{q/p}$, for $2\le i\le p$. The isomorphism respects the grading on $R_{\ep=1,s=0}$ given by the second grading on $R$ and the natural grading on $\cO_{q/p}$.
\end{lemma}

\subsection{Equivariant cohomology of $J_{q/p}$}\label{ss:eqcohJ} In Corollary~\ref{c:equiv gen} we showed that $\ep, \wt e_2,\cdots,\wt e_p$ generate $\eqcoh{*}{J_{q/p}}$. Recall the vector bundle $\Eun_{x}$ on $J_{q/p}\times \PP^{1}$. Let $\cE_{x}$ be the restriction of $\Eun_{x}$ to $J_{q/p}\times \{0\}=J_{q/p}$. Under the pullback along the homeomorphism $X_{q/p}\to J_{q/p}$, $\cE_{x}$ is the vector bundle whose fiber at $M\in X_{q/p}$ (viewed as a lattice in $\CC\lr{t}$) is $M/t^{p}M$. By definition (see Lemma~\ref{l:equiv gen}) we have $c^{\Gm}_{i}(\cE_{x})=p\ep \wt e_{i}, 1\le i\le p$. Similarly, using the $y$-projection on $C_{q/p}$, we get a vector bundle $\cE_{y}$ on $J_{q/p}$ whose pullback to $X_{q/p}$ has fiber $M/t^{q}M$ at $M$. We define $\wt f_{j}\in \eqcoh{2j-2}{J_{q/p}}$ such that $c^{\Gm}_{j}(\cE_{y})=q\ep f_{j}$, $1\le j\le q$.

\begin{prop}\label{p:eqcoh} 
\begin{enumerate}
\item There is a graded ring isomorphism
\begin{equation}\label{R eqcoh}
r: R_{s=1}\isom\eqcoh{*}{J_{q/p}}.
\end{equation}
which sends $\ep, e_{2},\cdots, e_{p}, f_{2},\cdots, f_{q}$ to $\ep,\wt e_{2},\cdots, \wt e_{p},\wt f_{2},\cdots, \wt f_{q}\in \eqcoh{*}{J_{q/p}}$ defined above. In particular, Theorem~\ref{th:HSp} holds, and we have an isomorphism
\begin{equation}\label{R spcoh}
r_{\ep=1}: R_{\ep=1, s=1}\isom\spcoh{}{J_{q/p}}.
\end{equation}
\item\label{no s torsion} The ring $R$ has no $s$-torsion.
\end{enumerate}
\end{prop}
\begin{proof} We first construct a surjective ring homomorphism $r: R_{s=1}\to \eqcoh{*}{J_{q/p}}$. We form two polynomials in $z$ whose coefficients are the equivariant Chern classes of $\cE_{x}$ and $\cE_{y}$
\begin{equation*}
\wt A(z)=z^{p}+\sum_{i=1}^{p}p\ep \wt e_{i} z^{p-i}; \quad \wt B(z)=z^{q}+\sum_{i=1}^{q}q\ep \wt f_{i} z^{q-i}.
\end{equation*}
Localizing to the fixed points and using \eqref{absum}, we see that
\begin{equation}\label{wt e1f1}
\wt e_{1}=q(p-1)/2;\quad \wt f_{1}=p(q-1)/2.
\end{equation}
Consider the $\Gm(q/p)$-equivariant vector bundle $\cE_{pq}$ on $X_{q/p}$ whose fiber at $M$ is $M/t^{pq}M$. Note that $M/t^{pq}M$ can be filtered in two ways
\begin{eqnarray*}
M/t^{pq}M\supset t^{p}M/t^{pq}M\supset\dots\supset t^{p(q-1)}M/t^{pq}M,\\
M/t^{pq}M\supset t^{q}M/t^{pq}M\supset\dots\supset t^{q(p-1)}M/t^{pq}M. 
\end{eqnarray*}
Therefore $\cE_{pq}$ has two filtrations by subbundles whose associated graded vector bundles are
\begin{eqnarray*}
\cE_{x}, \cE_{x}(p),\cdots, \cE_{y}(jp), \cdots, \cE_{x}(p(q-1));\\
\cE_{y}, \cE_{y}(q),\cdots, \cE_{y}(iq), \cdots, \cE_{x}(q(p-1)).
\end{eqnarray*} 
Here, for a $\Gm(q/p)$-equivariant bundle $\calE$, we denote by $\calE(1)$ by the same vector bundle with the $\Gm(q/p)$-equivariant structure twisted by the standard representation of $\Gm(q/p)$. Let $C(z)=z^{pq}+\sum_{k=1}^{pq}c^{\Gm}_{k}(\cE_{pq}) z^{pq-k}$. Then the Whitney formula for Chern classes gives that
\begin{equation}\label{wt AB}
\prod_{j=0}^{q-1}\wt A(z+jp\ep)= C(z)=\prod_{i=0}^{p-1}\wt B(z+iq\ep).
\end{equation}
Comparing the relations \eqref{wt e1f1}\eqref{wt AB} with \eqref{e1f1}\eqref{AB}, if we send the generators $\ep, e_{i}, f_{j}$ of $\wt R_{s=1}$ to $\ep, \wt e_{i}$ and $\wt f_{j}$ in $\eqcoh{*}{J_{q/p}}$, we get a ring homomorphism $\wt R_{s=1}\to \eqcoh{*}{J_{q/p}}$. Since $J_{q/p}$ is equivariantly formal, $\eqcoh{*}{J_{q/p}}$ has no $\ep$-torsion, this ring homomorphism factors through $r: R_{s=1}\to \eqcoh{*}{J_{q/p}}$. By Corollary~\ref{c:equiv gen}, $r$ is surjective. 

Localizing to the fixed points, we see that $\dim\spcoh{}{J_{q/p}}=\#\Sigma_{q/p}=\frac{1}{p+1}\binom{p+q}{p}$ by \eqref{fixed pts}. Therefore $\dim R_{\ep=1, s=1}\geq \frac{1}{p+1}\binom{p+q}{p}$. On the other hand, by \eqref{RO} and Lemma~\ref{l:Oqp}(4), we have
$$\dim R_{\ep=1,s=0}=\dim\cO_{q/p}=\frac{1}{p+q}\binom{p+q}{p}.$$ 
Therefore we have
\begin{equation*}
\dim R_{\ep=1,s=1}\geq\frac{1}{p+1}\binom{p+q}{p}=\dim \cO_{q/p}=\dim  R_{\ep=1,s=0}.
\end{equation*}
Since $R_{\ep=1,s=0}$ is a specialization of $R_{\ep=1,s=1}$, we must have the reversed inequality $\dim R_{\ep=1,s=1}\leq \dim  R_{\ep=1, s=0}$. Therefore all inequalities above are equalities. The equality $\dim R_{\ep=1,s=1}=\frac{1}{p+q}\binom{p+q}{p}=\dim \spcoh{}{J_{q/p}}$ implies that $r$ is an isomorphism because $R_{s=1}$ has no $\ep$-torsion. This proves part (1).

The equality $\dim R_{\ep=1,s=1}=\dim R_{\ep=1, s=0}$ implies that $R_{\ep=1}$ is flat over $\QQ[s]$, hence has no $s$-torsion. Now we show $R$ has no $s$-torsion. Suppose $r\in {}_{a}R_{b}$ is such that $sr=0$. Recall the canonical map $\iota_{b}: {}_{a}R_{b}\to {}_{\infty} R_{b}=R_{\ep=1}$ which is injective since $R$ has no $\ep$-torsion. Therefore $s\iota_{b}(r)=\iota_{b+1}(sr)=0$.  Since $R_{\ep=1}$ has no $s$-torsion, this implies $\iota_{b}(r)=0$ hence $r=0$. This proves part (2).
\end{proof}

\subsection{Cohomology of $\Sp_{q/p}$}\label{ss:coho aff flag}
Let $V=\CC\lr{x}^{p}$ and $\L^{\dagger}=\CC\tl{x}^{p}\subset V$ be the standard lattice. For each $i\in\ZZ$ we define $\Gr^{(i)}_{G}$ to be the moduli space of  $\CC\tl{x}$-lattices in $V$ such that $\dim(\L/\L\cap\L^{\dagger})=\dim(\L^{\dagger}/\L\cap\L^{\dagger})-(p-i)$. We have the affine Springer fiber $\Sp^{\Gr,(i)}_{q/p}\subset \Gr^{(i)}_{G}$ which classifies lattices $\L\subset V$ such that $\g_{q/p}\L\subset \L$ (see Section~\ref{Xqp ASF}). By identifying $V$ with $\CC\lr{t}$ as in Section~\ref{Xqp ASF}, $\Sp^{\Gr}_{q/p,i}$ can be identified with the moduli space $X^{(i)}_{q/p}$ which classifies $\CC\tl{t^{p},t^{q}}$-lattices $\L\subset \CC\lr{t}$ such that $\dim(\L/\L\cap \CC\tl{t^{p},t^{q}})=\dim(\CC\tl{t^{p},t^{q}}/\L\cap \CC\tl{t^{p},t^{q}})+i$. 

For $1\le i\le p$, we have a projection $\Pi_i:\Fl_{G}\rightarrow \Gr^{(i)}_{G}$ sending $\L_{\bullet}$ to $\L_{i}$, which restricts to a map
\begin{equation*}
\pi_{i}: \Sp_{q/p}\to X^{(i)}_{q/p}.
\end{equation*}
The torus $\Gm(q/p)$ acts on $V\cong\CC\lr{t}$ by scaling $t$, which induces an action on $\Gr^{(i)}_{G}$ and $X^{(i)}_{q/p}$. The maps $\Pi_{i}$ and $\pi_{i}$ are equivariant under $\Gm(q/p)$.

Recall from Section~\ref{sss:aff flag} that for each $i\in\ZZ$, we have a $\Grot$-equivariant line bundle $\calL_{i}$ on $\Fl$ whose fiber at $\L_{\bullet}\in \Fl_{G}$ is the line $\L_i/\L_{i-1}$. Denote by $\xi_i$ the $\Gm(q/p)$-equivariant Chern classes $c^{\Gm}_1(\calL_i)\in\eqcoh{2}{\Fl_{G}}$, for $i=1,\cdots,p$. 

On the other hand, $\Gr^{(i)}_{G}$ carries a $\Gm(q/p)$-equivariant rank $p$ vector bundle $\cE^{(i)}$ whose fiber at $\L\in\Gr^{(i)}_{G}$ is $\L/x\L$. The $\Gm(q/p)$-equivariant Chern classes $c^{\Gm}_j(\cE^{(i)})$ are divisible by the equivariant parameter $\ep$ (because in $\cohog{*}{\Gr^{(i)}_{G}}$ the usual Chern classes $c_{j}(\cE^{(i)})$ all vanish). We define
\begin{equation*}
e^{(i)}_{j}=p^{-1}\ep^{-1}c^{\Gm}_j(\cE^{(i)})\in \eqcoh{2j-2}{\Gr^{(i)}_{G}}.
\end{equation*}
Direct calculation shows that
\begin{equation*}
e^{(i)}_{1}=(p-1)q/2+p-i.
\end{equation*}
For $1\le j\le p$, let $e_j=\pi^{*}_{p}e^{(p)}_{j}\in \eqcoh{2j-2}{\Fl_{G}}$. The pullback 
\begin{equation*}
\Pi_{i}^{*}: \eqcoh{*}{\Gr^{(i)}_{G}}\to \eqcoh{*}{\Fl_{G}}
\end{equation*}
sends $e^{(i)}_{j}$ to $p^{-1}\ep^{-1}c_{j}(\xi_{i},\xi_{i-1},\cdots,\xi_{1}, \xi_{p}+p\ep,\cdots, \xi_{i+1}+p\ep)$, where  $c_{j}(...)$ is the $j$-th elementary symmetric polynomial of the $p$ elements obtained from $\x_{1},\cdots,\x_{p}$ by adding $p\ep$ to the last $p-i$ elements.

Multiplication by $t$ gives an isomorphism $\Gr^{(i)}_{G}\cong\Gr^{(i-1)}_{G}$ which restricts to an isomorphism $X^{(i)}_{q/p}\cong X^{(i-1)}_{q/p}$. Therefore $X^{(i)}_{q/p}$ can be identified with $X_{q/p}=\SG$ for all $i$. In particular, Theorem~\ref{th:gen} implies that the natural map
\begin{equation*}
\QQ[\ep, e^{(i)}_{2},\cdots, e^{(i)}_{p}]\to \eqcoh{*}{X^{(i)}_{q/p}}
\end{equation*}
is surjective. We denote the kernel of this map by $I^{(i)}_{q/p}$. When $i=p$, the ideal $I^{(p)}_{q/p}$ is indirectly described in Theorem~\ref{th:HSp} if we first express $f_{j}$ using $e_{j}$ and then substituting into $I^{sat}_{q/p}$. For other $i$, the ideal $I^{(i)}_{q/p}$ is obtained from  $I^{(p)}_{q/p}$ by substituting $e^{(p)}_{j}$ by $e^{(i)}_{j}$. Clearly the restriction of $\Pi^{*}_{i}I^{(i)}_{q/p}$ to $\Sp_{q/p}$ is zero.

Inspired by the work of Tanisaki \cite{Tanisaki} and Abe-Horiguchi \cite{AH} we  propose:

\begin{conj}\label{conj:rel Sp Fl} The kernel of the surjection $\QQ[\ep,\xi_{1},\cdots,\xi_{p}, e_{2},\cdots,e_{p}]\surj \eqcoh{*}{\Sp_{q/p}}$ in Theorem~\ref{th:par gen} is generated by $\pi_i^*(I^{(i)}_{q/p})$ for $i=1,\cdots, p$.
\end{conj}

The following proposition shows that the conjecture holds up to $\ep$-saturation.  If we specialize $\ep=1$, $e_{i}$ are expressible as symmetric polynomials of $\xi_{1},\cdots, \xi_{p}$, so $\xi_{1},\cdots,\xi_{p}$ generate $\spcoh{}{\Sp_{q/p}}$ as a $\QQ$-algebra. We denote the image of $\pi_i^*(I^{(i)}_{q/p})$ under the specialization $\ep=1$ by $\pi_i^*(I^{(i)}_{q/p})_{\ep=1}\subset \QQ[\xi_{1},\cdots,\xi_{p}]$.

\begin{prop}\label{p:par rel} The kernel of the surjection $\QQ[\xi_{1},\cdots,\xi_{p}]\surj \spcoh{}{\Sp_{q/p}}$ is generated by $\pi_i^*(I^{(i)}_{q/p})_{\ep=1}$ for $i=1,\cdots, p$. 

In particular, the kernel of $\QQ[\ep,\xi_{1},\cdots,\xi_{p}, e_{2},\cdots,e_{p}]\surj \eqcoh{*}{\Sp_{q/p}}$ is the $\ep$-saturation of the ideal generated by $\pi_i^*(I^{(i)}_{q/p})$ for $i=1,\cdots, p$.
\end{prop}
\begin{proof}
Recall the bijection between the set $\wt\Sigma_{q/p}$ and the $\Gm(q/p)$-fixed points on $\Sp_{q/p}$ in Section~\ref{sss:aff flag}. Localizing to the $\Gm(q/p)$-fixed points, we get a map
\begin{equation}\label{localize Fl}
\rho: \QQ[\xi_{1},\cdots,\xi_{p}]\to \spcoh{}{\Sp_{q/p}}\isom \oplus_{\un d\in \wt\Sigma_{q/p}}\QQ
\end{equation}
whose $\un d$-component is given by evaluation $(\xi_{1},\cdots,\xi_{p})\mapsto (d_{1},\cdots,d_{p})$. 

We have the following analogue of the discussion in Section~\ref{FixedPoints}. For $1\le i\le p$, the $\Gm(q/p)$-fixed points on $X_{q/p}^{(i)}$ are in natural bijection with $\Sigma^{(i)}_{q/p}$, the set of $\jiao{p,q}$-modules $\s\subset \ZZ$ such that $\#(\s-\s\cap\jiao{p,q})=\#(\jiao{p,q}-\s\cap\jiao{p,q})-(p-i)$. Localizing the cohomology of $X^{(i)}_{q/p}$ to fixed points, we have a ring isomorphism
\begin{equation}\label{localize Gr}
\QQ[e^{(i)}_{2},\cdots, e^{(i)}_{p}]/I^{(i)}_{q/p,\ep=1}\isom \spcoh{}{X^{(i)}_{q/p}}\cong \oplus_{\s\in\Sigma^{(i)}_{q/p}}\QQ
\end{equation}
whose $\s$-component sends $e^{(i)}_{j}$ to the $j$-th elementary symmetric polynomial of the $p$-basis of $\s$.

Let $I'\subset\QQ[\xi_{1},\cdots,\xi_{p}]$ be the ideal generated by $\pi_i^*(I^{(i)}_{q/p})_{\ep=1}$ for $i=1,\cdots, p$. Let $Z=\Spec \QQ[\xi_{1},\cdots,\xi_{p}]/I']\subset \AA^{p}$. We have a map $\gamma_{i}: \AA^{p}\to \AA^{p}$ sending $(a_{1},\cdots,a_{p})$ to the elementary symmetric polynomials of $(a_{i},a_{i-1},\cdots,a_{1}, a_{p}+p,\cdots, a_{i+1}+p)$. Let $Z_{i}=\QQ[e^{(i)}_{2},\cdots, e^{(i)}_{p}]/I^{(i)}_{q/p,\ep=1}$ viewed as a subscheme of $\AA^{p}$ under coordinates $e^{(i)}_{1},\cdots, e^{(i)}_{p}$. By construction we have 
\begin{equation*}
Z=\cap_{i=1}^{p}\gamma_{i}^{-1}(Z_{i})
\end{equation*}
as a subscheme of $\AA^{p}$. By \eqref{localize Gr} we see each $Z_{i}$ is a finite set of reduced points which are away from the branching locus of $\gamma_{i}$ (because the $p$-basis elements of $\s\in \Sigma^{(i)}_{q/p}$ are distinct). Therefore $Z$ is also reduced. Any $\CC$-point $\un a\in Z$ must map to some $\s_{i}\in\Sigma_{q/p}^{(i)}$ under $\gamma_{i}$. The set of such $\un a\in \CC^{p}$ is exactly $\wt\Sigma_{q/p}$. Therefore $Z$ is the reduced subscheme of $\AA^{p}$ given by $\wt\Sigma_{q/p}$. By \eqref{localize Fl}, we have $Z=\Spec\spcoh{}{\Sp_{q/p}}$, hence $I'$ generates the defining ideal of $\spcoh{}{\Sp_{q/p}}$.
\end{proof}

\section{Filtrations on cohomology}

The goal of this section is prove Theorems~\ref{th:perv mult}, \ref{th:main} and \ref{th:Gr m}.

\subsection{Filtrations on $R$} Recall the bigraded ring $R$ from Section~\ref{sss:R}.

\subsubsection{Double filtration}\label{sss:double fil} For $b\in\ZZ_{\geq0}$, let $_{\infty}R_{b}=\varinjlim_{a}{}_{a}R_{b}$ where the transition map is given by $\ep$. Then we have a canonical graded isomorphism $\oplus_{b}({}_{\infty}R_{b})=R/(\ep-1)R=R_{\ep=1}$ where $R_{\ep=1}$ inherits the second grading from $R$. Multiplication by $s$ induces a map ${}_{\infty}R_{b}\to {}_{\infty}R_{b+1}$, and let ${}_{\infty}R_{\infty}=\varinjlim_{b}{}_{\infty}R_{b}$ with transition map $s$.  We have a canonical ring isomorphism ${}_{\infty}R_{\infty}\cong R/(\ep-1,s-1)R=R_{\ep=1,s=1}$.

Similarly, for $a\in\ZZ_{\geq0}$ let $_{a}R_{\infty}=\varinjlim_{b}{}_{a}R_{b}$ where the transition map is given by $s$, then we have a canonical graded isomorphism $\oplus_{a}({}_{a}R_{\infty})=R/(s-1)R=R_{s=1}$ (which inherits the first grading). Multiplication by $\ep$ induces a map ${}_{a}R_{\infty}\to {}_{a+1}R_{\infty}$, and $\varinjlim_{a}{}_{a}R_{\infty}$ with transition maps $\ep$ is canonically isomorphic to ${}_{\infty}R_{\infty}$ defined above, which is isomorphic to $R_{\ep=1,s=1}$.

By definition, $R$ has no $\ep$-torsion; by Proposition~\ref{p:eqcoh}\eqref{no s torsion}, $R$ also has no $s$-torsion. Therefore the natural maps ${}_{a}R_{\infty}\to {}_{\infty}R_{\infty}$ and ${}_{\infty}R_{b}\to {}_{\infty}R_{\infty}$ are injective. We define the $F^{\ep}$-filtration on ${}_{\infty}R_{\infty}=R_{\ep=1,s=1}$ by
\begin{equation*}
F^{\ep}_{\leq a}R_{\ep=1,s=1}=\Im({}_{a}R_{\infty}\to {}_{\infty}R_{\infty})\cong {}_{a}R_{\infty}.
\end{equation*}
Similarly, we define the $F^{s}$-filtration on ${}_{\infty}R_{\infty}=R_{\ep=1,s=1}$ by
\begin{equation*}
F^{s}_{\leq b}R_{\ep=1,s=1}=\Im({}_{\infty}R_{b}\to {}_{\infty}R_{\infty})\cong {}_{\infty}R_{b}.
\end{equation*}
Whenever we have two filtrations, we have a canonical isomorphism between bigraded rings by taking associated graded of the two filtrations in different orders
 \begin{equation}\label{two fil}
 \Gr^{F^{\ep}}_{i}\Gr^{F^{s}}_{j}R_{\ep=1,s=1}\cong \Gr^{F^{s}}_{j}\Gr^{F^{\ep}}_{i}R_{\ep=1,s=1}.
\end{equation}

By definition, we always have an inclusion
\begin{equation}\label{incl Rab}
{}_{a}R_{b}\subset {}_{a}R_{\infty}\cap{}_{\infty}R_{b}.
\end{equation}
When $R$ is flat over $\QQ[\ep,s]$, this inclusion becomes an equality.

\begin{lemma}\label{l:Gr Fs O} There is a graded ring isomorphism $\Gr^{F^{s}}_{*}R_{\ep=1,s=1}\cong \cO_{q/p}$. Under this isomorphism, we transport the filtration $F^{\ep}$ on $\Gr^{F^{s}}_{*}R_{\ep=1,s=1}$ to $\cO_{q/p}$. We have
\begin{equation*}
F^{\ep}_{\le i}\cO_{q/p}[j]\supset \fm^{j-i}\cap\cO_{q/p}[j], \quad \forall  j\ge i\ge0.
\end{equation*}
Here $\fm\subset\cO_{q/p}$ is the maximal ideal. 
\end{lemma}
\begin{proof}
Since $R$ has no $s$-torsion by Proposition~\ref{p:eqcoh}(2), neither does $R_{\ep=1}$. Therefore $\Gr^{F^{s}}_{*}R_{\ep=1,s=1}=R_{\ep=1,s=0}$. By \eqref{RO}, we have $R_{\ep=1,s=0}=\cO_{q/p}$. Therefore we get the desired graded ring isomorphism $\Gr^{F^{s}}_{*}R_{\ep=1,s=1}\cong \cO_{q/p}$.

Using notation from Section~\ref{sss:double fil}, we have a graded isomorphism
\begin{equation*}
\cO_{q/p}\cong R_{\ep=1,s=0}\cong \bigoplus_{j}{}_{\infty}R_{j}/{}_{\infty}R_{j-1}.
\end{equation*}
The induced filtration $F^{\ep}$ on $\cO_{q/p}$ is then
\begin{equation*}
F^{\ep}_{\le i}\cO_{q/p}[j]=({}_{i}R_{\infty}\cap {}_{\infty}R_{j})/({}_{i}R_{\infty}\cap{}_{\infty}R_{j-1}).
\end{equation*}
By the inclusion \eqref{incl Rab},  we see that 
\begin{equation*}
F^{\ep}_{\le i}\cO_{q/p}[j]\supset \textup{Im}({}_{i}R_{j}\to R_{\ep=1,s=0}).
\end{equation*}
By construction, ${}_{i}R_{j}$ is spanned by monomials $\ep^{m}e_{2}^{i_{2}}\cdots e_{p}^{i_{p}}$ such that $\sum_{k=2}^{p}ki_{k}=j$ and $m+\sum_{k=2}^{p}(k-1)i_{k}=i$. This implies $i_{2}+\cdots+i_{p}=j-i-m\ge j-i$. Therefore the image of $e_{2}^{i_{2}}\cdots e_{p}^{i_{p}}$ in $R_{\ep=1,s=0}$ belongs to $F^{\ep}_{\le i}\cO_{q/p}[j]$ if $\sum_{k=2}^{p}ki_{k}=j$ and $i_{2}+\cdots+i_{p}\ge j-i$. The latter span exactly $\fm^{j-i}\cap\cO_{q/p}[j]$.
\end{proof}

\subsection{The Chern filtration}
As in \cite[Definition 8.2.1]{OY}, we make the following definitions.
\begin{defn}
\begin{enumerate}
\item The {\em Chern filtration} $C_{\le n}\spcoh{}{J_{q/p}}$ on the specialized equivariant cohomology $\spcoh{}{J_{q/p}}$ is the span of the image of monomials $e^{a_{2}}_{2}\cdots e^{a_{p}}_{p}$ (under the isomorphism \eqref{R spcoh}) such that $\sum_{i=2}^{p}ia_{i}\leq n$. 
\item The {\em Chern filtration} $C_{\le n}\eqcoh{*}{J_{q/p}}$ on the equivariant cohomology $\eqcoh{*}{J_{q/p}}$ is the $\QQ[\ep]$-saturation of $C_{\le n}\spcoh{}{J_{q/p}}$, i.e., $h\in \eqcoh{*}{J_{q/p}}$ lies in $C_{\le n}\eqcoh{*}{J_{q/p}}$ if and only if $h|_{\ep=1}\in C_{\le n}\spcoh{}{J_{q/p}}$.
\item The {\em Chern filtration} $C_{\le n}\cohog{*}{J_{q/p}}$ on the usual cohomology $\cohog{*}{J_{q/p}}$ is the image of $C_{\le n}\eqcoh{*}{J_{q/p}}$ in $\cohog{*}{J_{q/p}}$.
\end{enumerate}
\end{defn}

\begin{remark} There is a more native candidate for the Chern filtration on $\eqcoh{*}{J_{q/p}}$: simply take the span of the Chern monomials $\ep^{m}\wt e^{a_{2}}_{2}\cdots \wt e^{a_{p}}_{p}$ such that $\sum_{i=2}^{p}ia_{i}\leq n$. However, it is not clear why this span is $\ep$-saturated. Our definition of the Chern filtration is the $\QQ[\ep]$-saturation of the naive one, which is also the strategy used in \cite[Section~8.2]{OY}.
\end{remark}

The filtration $F^{s}_{\leq n}$ on $R_{\ep=1, s=1}$ induces a filtration on $R_{s=1}$ by $\ep$-saturation: $h\in R_{s=1}$ lies in $F^{s}_{\le n}R_{s=1}$ if and only if its image in $R_{\ep=1,s=1}$ lies in $F^{s}_{\leq n}R_{\ep=1, s=1}$. We then define a filtration $F^{s}_{\le n}$ on $R_{\ep=0,s=1}$ by taking the image of  $F^{s}_{\le n}R_{s=1}$.  

\begin{lemma} Under the isomorphism $R_{s=1}\cong \eqcoh{*}{J_{q/p}}$ established in Proposition~\ref{p:eqcoh}, we have equalities of filtrations
\begin{eqnarray}
\label{C=Fs spcoh} C_{\le n}\spcoh{}{J_{q/p}}&=&F^{s}_{\le n}R_{\ep=1,s=1};\\
\label{C=Fs eqcoh} C_{\le n}\eqcoh{*}{J_{q/p}}&=&F^{s}_{\le n}R_{s=1};\\
\label{C=Fs coh} C_{\le n}\cohog{*}{J_{q/p}}&=&F^{s}_{\le n}R_{\ep=0, s=1}.
\end{eqnarray}
\end{lemma}
\begin{proof}
It suffices to prove \eqref{C=Fs spcoh}: for \eqref{C=Fs eqcoh},  both sides are defined as the saturations of the corresponding sides in \eqref{C=Fs spcoh}; for \eqref{C=Fs coh}, both sides are obtained by taking images of the corresponding sides of \eqref{C=Fs eqcoh}.

We show \eqref{C=Fs spcoh}. By definition, $F^{s}_{\le n}R_{\ep=1,s=1}={}_{\infty}R_{n}$. Since ${}_{a}R_{b}$ is spanned by monomials of bidegree $(a,b)$,   $F^{s}_{\le n}R_{\ep=1,s=1}$ is spanned by the image of monomials of bidegrees $(a,n)$ for all $a\in\ZZ$, which corresponds exactly to $C_{\le n}\spcoh{}{J_{q/p}}$.
\end{proof}

\subsection{The perverse filtration}\label{sss:perv}
The perverse filtration on the cohomology of the compactified Jacobian of a locally planar curve $C$ is introduced in \cite{MY}. Let us recall the definition. 

Let $C$ be an integral curve $C$ with planar singularities.  Let $\pi: \cC\to \cB$ be a locally projective flat deformation of $C$ over an irreducible base $\cB$: say the point $b_{0}\in \cB$ is such that $\cC_{b_{0}}=C$. Assume $\pi$ satisfies the following conditions
\begin{enumerate}
\item Each geometric fiber $\cC_{b}$ ($b\in \cB$) is integral and has only planar singularities. Let $g_{a}$ denote the arithmetic genus of all geometric fibers $\cC_{b}$. 
\item For each $0\leq n\leq 2g_{a}-1$, the total space of the relative Hilbert scheme $\Hilb^{n}(\cC/\cB)$ is smooth.
\item For every (not necessarily closed) point $b\in \cB$, we have the $\d$-invariant $\d(b)$ of the fiber $\cC_{b}$ (defined as the length of $\nu_{*}\calO_{\wt\cC_{b}}/\calO_{\cC_{b}}$, where $\nu:\wt\cC_{b}\to\cC_{b}$ is the normalization map). Then we have
$\codim_{\cB}(\ov{\{b\}})\geq \d(b)$, where $\ov{\{b\}}$ is the Zariski closure of $b$ in $\cB$.
\end{enumerate}
Replacing $\cB$ by an \'etale neighborhood of $b_{0}$, we may assume that $\pi$ admits a section, hence the relative compactified Jacobian $\calJ=\cJ(\cC/\cB)$ is defined and its total space is smooth. Let $f:\calJ\to \cB$ be the natural map. Then the complex $\bR f_{*}\QQ$ carries the perverse truncations $^{p}{\tau}_{\leq i}\bR f_{*}\QQ$. The stalk $(\bR f_{*}\QQ)_{b_{0}}$ is canonically isomorphic to the cohomology of $\cJ(C)$. We define an increasing filtration  $P_{\le i}$ on $\cohog{*}{\cJ(C)}$ by
\begin{equation}\label{define perv}
P_{\leq i}\cohog{*}{\cJ(C)}=\textup{Im}((^{p}{\tau}_{\leq i+\dim\cB}\bR f_{*}\QQ)_{b_{0}}\to (\bR f_{*}\QQ)_{b_{0}}).
\end{equation}
It was shown in \cite[Theorem 1.1]{MY} that the filtration $P_{\leq i}\cohog{*}{\cJ(C)}$ is independent of the choice of the deformation $\cC\to \cB$ satisfying the above conditions.   We call the canonical filtration $P_{\leq i}\cohog{*}{\cJ(C)}$ thus obtained the {\em perverse filtration} on $\cohog{*}{\cJ(C)}$.  We have $\Gr^{P}_{j}\cohog{i}{\cJ(C)}=0$ unless $i/2\leq j\leq i$. 

Now consider the family of spectral curves $\cY\to \cA'$ associated to the Hitchin moduli stack $\cM$ in Section~\ref{sss:ell}. It is shown in \cite[Proposition 3.3]{MY} that $\cY$ satisfies the conditions above ({\it loc.cit.} proves the case where the base curve $X$ is a scheme, but the argument there works for Deligne-Mumford curves such as $\PP(p,1)$). By Lemma~\ref{l:ell}, $f':M'\to \cA'$ can be identified with the family of compactified Jacobians $\cJ(\cY/\cA')$. The perverse filtration on $\cohog{*}{J_{q/p}}\cong\cohog{*}{M_{1}}$  then can be computed using \eqref{define perv} with $f=f'$. Since the family $f':M'\to \cA'$ carries a $\Gm(q/p)$-action and $1\in \cA_{q/p}$ is fixed by this action, the same recipe as in \eqref{define perv} also defines a perverse filtration on the equivariant cohomology $\eqcoh{*}{J_{q/p}}$ and its specialization.

\begin{prop}\label{prp:perv_Chern}
There are equalities of filtrations:
\begin{eqnarray}
\label{C=P spcoh} C_{\le n}\spcoh{}{J_{q/p}}&=&P_{\le n}\spcoh{}{J_{q/p}};\\
\label{C=P eqcoh} C_{\le n}\eqcoh{*}{J_{q/p}}&=&P_{\le n}\eqcoh{*}{J_{q/p}};\\
\label{C=P coh} C_{\le n}\cohog{*}{J_{q/p}}&=&P_{\le n}\cohog{*}{J_{q/p}}. 
\end{eqnarray}
\end{prop}
\begin{proof}
We first show that both \eqref{C=P coh} and \eqref{C=P eqcoh} follow from \eqref{C=P spcoh}.

The equivariant cohomology $\eqcoh{*}{J_{q/p}}\cong \eqcoh{*}{M_{1}}$ is computed as the $\Gm(q/p)$-equivariant stalk of $\bR f'_{!}\QQ$ at $a=(0,\cdots,0,-\x^{q})\in\cA_{q/p}$. Let $d=\dim\cA'$. By definition, 
\begin{equation*}
P_{\le n}\eqcoh{*}{J_{q/p}}=\textup{Im}(\cohog{*}{[\pt/\Gm], (\ptau_{\le n+d}\bR f'_{!}\QQ)_{a}}\xr{i} \cohog{*}{[\pt/\Gm],(\bR f'_{!}\QQ})_{a}).
\end{equation*}
By the decomposition theorem applied to $f'$ (which is proper by Lemma~\ref{l:ell}\eqref{f proper} with $M'$ smooth by Proposition~\ref{p:sm}), or rather the induced map of stalks $[M'/\Gm(q/p)]\to [\cA'/\Gm(q/p)]$, $\ptau_{\le n+d}\bR f'_{!}\QQ$ is a direct summand of $\bR f'_{!}\QQ$ in the $\Gm$-equivariant derived category, therefore the arrow $i$ in the above equation is also the inclusion of a direct summand as a $\QQ[\ep]$-module. In particular, we see that $P_{\le n}\eqcoh{*}{J_{q/p}}$ is $\ep$-saturated, and we have 
\begin{equation}\label{perv fil eqcoh}
P_{\le n}\eqcoh{*}{J_{q/p}}\cong \cohog{*}{[\pt/\Gm], (\ptau_{\le n+d}\bR f'_{!}\QQ)_{a}}.
\end{equation}
Since $P_{\le n}\eqcoh{*}{J_{q/p}}$ is $\ep$-saturated,  it is the same as the $\ep$-saturation of the induced perverse filtration on $\spcoh{}{J_{q/p}}$. The same relation holds for the Chern filtrations on $\eqcoh{*}{J_{q/p}}$ and $\spcoh{}{J_{q/p}}$. Therefore,  \eqref{C=P eqcoh} follows from \eqref{C=P spcoh}.

For the usual cohomology, the perverse filtration on $\cohog{*}{J_{q/p}}\cong (\bR f'_{!}\QQ)_{a}$ is defined by
\begin{equation}\label{perv fil coh}
P_{\le n}\cohog{*}{J_{q/p}}\cong (\ptau_{\le n+d}\bR f'_{!}\QQ)_{a}.
\end{equation}
Comparing \eqref{perv fil eqcoh} and \eqref{perv fil coh}, we get  $P_{\le n}\cohog{*}{J_{q/p}}\cong P_{\le n}\eqcoh{*}{J_{q/p}}\otimes_{\QQ[\ep]}\QQ$ (specializing $\ep$ to $0$). The same relation $C_{\le n}\cohog{*}{J_{q/p}}= C_{\le n}\eqcoh{*}{J_{q/p}}\otimes_{\QQ[\ep]}\QQ$ holds by definition. Therefore,  \eqref{C=P coh} follows from \eqref{C=P eqcoh}.

Now it remains to show \eqref{C=P spcoh}. We use the notation and discussion from Section~\ref{ss:coho aff flag} on the cohomology of $\Sp_{q/p}$.  Recall we have a canonical surjection
\begin{equation}\label{Chern aff flag}
\QQ[\xi_{1},\cdots,\xi_{p}]\surj \spcoh{}{\Sp_{q/p}}
\end{equation}
sending $\xi_{i}$ to $c^{\Gm}_{1}(\cL_{i})$ for $1\leq i\leq p$. In \cite[Definition 8.2.1]{OY}, we defined the Chern filtration on $C_{\leq i}\spcoh{}{\Sp_{q/p}}$ to be the image of polynomials in $\xi_{1},\cdots, \xi_{p}$ of degree $\le i$. In \cite[Definition 8.5.1]{OY}, we have defined a perverse filtration on $\eqcoh{*}{\Sp_{q/p}}$  (note the group $\wt S_{a}$ that appears in the definition there acts trivially on $\eqcoh{*}{\Sp_{q/p}}$ in the case $a=\g_{q/p}$). Recall the construction in \cite{OY} uses the parabolic Hitchin fibration $f^{\pb}:\cM^{\pb}_{p,\D}(X,\cL)\to \cA$. The perverse filtration on $\eqcoh{*}{\Sp_{q/p}}$ induces one on $\spcoh{}{\Sp_{q/p}}$. By \cite[Proposition 8.5.4]{OY}, 
\begin{equation}\label{C=P spcoh aff fl}
C_{\le n}\spcoh{}{\Sp_{q/p}}=P_{\le n}\spcoh{}{\Sp_{q/p}}.
\end{equation}

The $\Gm(q/p)$-equivariant projection $\pi: \Sp_{q/p}\to \SG$ (sending a chain $\L_{\bullet}$ to $\L_{p}$) induces an inclusion
\begin{equation*}
\pi^{*}_{\ep=1}: \spcoh{}{\SG}\incl \spcoh{}{\Sp_{q/p}}.
\end{equation*}
In view of \eqref{C=P spcoh aff fl}, \eqref{C=P spcoh} would follow the following two claims:
\begin{eqnarray}
\label{P strict} P_{\leq n}\spcoh{}{\SG}=\pi^{*,-1}_{\ep=1} P_{\leq n}\pi^{*} \spcoh{}{\Sp_{q/p}};\\
\label{C strict} C_{\leq n}\spcoh{}{\SG}=\pi^{*,-1}_{\ep=1}C_{\leq n} \spcoh{}{\Sp_{q/p}}.
\end{eqnarray}
We first show \eqref{P strict}. By \cite[Section~8]{OY} that $\spcoh{}{\Sp_{q/p}}$ is an irreducible representation of the bigraded rational Cherednik algebra $\Hrat$ (we refer to \cite{OY} for details). In particular, the symmetric group  $S_p$ acts on $\eqcoh{*}{\Sp_{q/p}}$, and this action preserves the perverse filtration (see \cite[Construction 8.4.2]{OY}). Since the map $\pi^*_{\ep=1}$ identifies $ \spcoh{}{\SG}$ with the $S_{p}$-invariants on $\spcoh{}{\Sp_{q/p}}$,  \eqref{P strict} follows.

We now show \eqref{C strict}. We know that $\pi^{*}_{\ep=1}e_{i}\in\spcoh{}{\Sp_{q/p}}$ is the $i$-th elementary symmetric polynomial in $\xi_{1},\cdots, \xi_{p}$.  Since $C_{\leq n} \eqcoh{*}{\Sp_{q/p}}$ is spanned by polynomials of $\xi_{1},\cdots,\xi_{p}$ of degree $\le n$, $C_{\leq n}\eqcoh{*}{\SG}$ is spanned by polynomials of $e_{2},\cdots, e_{p}$ of degree $\le n$, therefore \eqref{C strict} holds.
\end{proof}

\sss{Proof of Theorem~\ref{th:perv mult}}
By definition, the Chern filtration on $\cohog{*}{J_{q/p}}$ are multiplicative. Since the Chern filtration on $\cohog{*}{J_{q/p}}$ coincides with the perverse filtration by Proposition~\ref{prp:perv_Chern}, the perverse filtration is also multiplicative.

\subsection{Proof of Theorem~\ref{th:main}} 

(1) We take the filtration $F$ on $\cO_{q/p}$ to be $F^{\ep}$.  By Lemma~\ref{l:Gr Fs O}, we have
\begin{equation*}
\Gr^{F^{\ep}}_{i}\cO_{q/p}[j]\cong\Gr^{F^{\ep}}_{i}\Gr^{F^{s}}_{j}R_{\ep=1, s=1}.
\end{equation*}
Combining \eqref{C=P coh} and \eqref{C=Fs coh}, we have
\begin{equation*}
\Gr^{P}_{j}\cohog{i}{J_{q/p}}=\Gr^{C}_{j}\cohog{i}{J_{q/p}}\cong \Gr^{F^{s}}_{j}R_{\ep=0,s=1}=\Gr^{F^{s}}_{j}\Gr^{F^{\ep}}_{i}R_{\ep=1,s=1}.
\end{equation*}
Finally we use the observation \eqref{two fil} to see that the RHS of the above two equations are the same, therefore
\begin{equation}\label{two graded}
\Gr^{P}_{j}\cohog{2i}{J_{q/p}}\cong \Gr^{F^{\ep}}_{i}\cO_{q/p}[j].
\end{equation}

(2) The containment \eqref{F contains m} is proved in Lemma~\ref{l:Gr Fs O}. 

The Hard Lefschetz theorem for the perverse filtration on  \(\cohog{*}{J_{q/p}}\) implies that cupping with $e_{2}^{\d-j}$ induces an isomorphism 
\begin{equation}\label{eq:Hard L}
\Gr_j^P \cohog{2i}{J_{q/p}}\isom\Gr^P_{2\delta-j}\cohog{2(\delta-j+i)}{J_{q/p}}.
\end{equation}

Suppose the equality in \eqref{F contains m} holds, then Conjecture \ref{conj:Gr m} holds, hence
\begin{eqnarray*}
\dim \textup{H}^{2i}(J_{q/p})&=&\sum_j \dim \Gr^P_j\textup{H}^{2i}(J_{q/p})\\
&=&\sum_j\dim \Gr^P_{2\delta-j}\textup{H}^{2(\delta-j+i)}(J_{q/p}) \quad \mbox{(by \eqref{eq:Hard L})}\\
&=& \sum_j\dim \Gr_{\mathfrak{m}}^{\delta-i}\mathcal{O}_{q/p}[2\delta-j] \quad \mbox{(by Conjecture \ref{conj:Gr m})}\\
&=&\dim \Gr_{\mathfrak{m}}^{\delta-i}\mathcal{O}_{q/p},
\end{eqnarray*}
which is the equality \eqref{H2i Grm}. 
 
Conversely, suppose \eqref{H2i Grm} holds for all $i$. We claim that \eqref{F contains m} must be an equality. Let
\begin{equation*}
a_{ij}=\dim\Gr^{F^{\ep}}_{\d-i}\cO_{q/p}[2\d-j], \quad b_{ij}=\dim \Gr^{i}_{\fm}\cO_{q/p}[j].
\end{equation*}
The equality \eqref{H2i Grm} implies that
\begin{equation}\label{col sum}
\sum_{j}a_{ij}=\sum_{j}b_{ij}.
\end{equation}Hard Lefschetz \eqref{eq:Hard L} implies $a_{ij}=a_{\d-j+i,2\d-j}$. The containment \eqref{F contains m} implies that for any $j$
\begin{eqnarray}\label{ab ineq}
&&\sum_{i'\ge i}b_{i',j}=\dim \fm^{i}\cap \cO_{q/p}[j]\le \dim F^{\ep}_{\le j-i}\cO_{q/p}[j]\\
\notag&=&\sum_{i'\ge i}\dim \Gr^{F^{\ep}}_{j-i'}\cO_{q/p}[j]=\sum_{i'\ge i}a_{\d-j+i',2\d-j}=\sum_{i'\ge i}a_{i',j}.
\end{eqnarray}
We prove by descreasing induction on $i$ that $a_{ij}=b_{ij}$ for all $j$. If $i>\d$, then $\sum_{j}a_{ij}=\dim \cohog{2i}{J_{q/p}}=0$ which implies $a_{ij}=0$ hence $b_{ij}=0$ by \eqref{col sum}. Now suppose $a_{i'j}=b_{i'j}$ holds for all $i'> i$ and all $j$, then \eqref{ab ineq} implies that $a_{ij}\ge b_{ij}$ for all $j$. Combined with \eqref{col sum} we see that $a_{ij}=b_{ij}$. This shows that $a_{ij}=b_{ij}$ for all $i,j\in\ZZ$. This forces the inequality in \eqref{ab ineq} to be an equality, which in turn implies that the containment \eqref{F contains m} is an equality. The proof of Theorem \ref{th:main} is complete.

\subsection{Conjectural flatness of $R$} Numerical evidence seems to support the following

\begin{conj}
The algebra $R$ is flat over $\QQ[\ep,s]$.
\end{conj}

If conjecture is true then the $\QQ[\ep,s]$-algebra $R$ is a bigraded algebra whose various specializations are related to $\cO_{q/p}$ and $\cohog{*}{J_{q/p}}$ in the following way (where $\leadsto$ denotes degeneration, or equivalently passing to the associated graded with respect to a certain filtration)
\begin{equation*}
\xymatrix{ R_{\ep=1,s=1}=\spcoh{*}{J_{q/p}}\ar@{~>}[d] \ar@{~>}[r] & R_{\ep=0, s=1}=\cohog{*}{J_{q/p}}\ar@{~>}[d] \\
R_{\ep=1,s=0}=\cO_{q/p}\ar@{~>}[r] & \Gr^{*}_{\fra}\cO_{q/p}\cong \Gr^{P}_{*}\cohog{*}{J_{q/p}}}
\end{equation*}

\sss{Proof of Theorem~\ref{th:Gr m}}
Examining the proof of Lemma~\ref{l:Gr Fs O}, the only place that caused the containment rather than equality is the use of \eqref{incl Rab}. If $R$ is flat over $\QQ[\ep,s]$, \eqref{incl Rab} is an equality and the argument of Lemma~\ref{l:Gr Fs O} then shows that $F^{\ep}_{\le i}\cO_{q/p}[j]=\fm^{j-i}\cap \cO_{q/p}[j]$ for $j\ge i\ge0$. For $j<i$ we have the trivial equality $F^{\ep}_{\le i}\cO_{q/p}[j]=\cO_{q/p}[j]$. Therefore 
\begin{equation*}
\Gr^{F^{\ep}}_{i}\cO_{q/p}[j]= \Gr^{j-i}_{\fm}\cO_{q/p}[j]
\end{equation*}
for all $i$ and $j$. Theorem~\ref{th:Gr m} then follows from \eqref{two graded}.

\section{Relation with \Cha and Hilbert scheme}
\label{sec:Hilb}

In this section we discuss another approach to Conjecture~\ref{conj:Gr m} in the case \(q=kp+1\) by making connections with the rational Cherednik algebras and the Hilbert scheme points on $\CC^{2}$.

\subsection{Rational \Cha}
\subsubsection{}
Let $\frh$ be the Cartan subaglebra of $\sl_p$ and $\frh^*$ be its dual. Let $W=S_{p}$ be the Weyl group.
The bigraded rational \Cha $\Hrat_\nu$ is, as a vector space, a tensor product
\begin{equation*}
\Hrat=\QQ[\epsilon]\otimes\Sym(\frh)\otimes\Sym(\frh^{*})\otimes\QQ[W]\otimes\QQ[\L_{\can}].
\end{equation*}
The bigrading is given by
\begin{eqnarray*}
&&\deg(\epsilon)=(2,0);\\
&&\deg(\eta)=(0,-1)\hspace{1cm}\forall\eta\in\frh;\\
&&\deg(\xi)=(2,1)\hspace{1cm}\forall\xi\in\frh^{*};\\
&&\deg(w)=(0,0)\hspace{1cm}\forall w\in W;
\end{eqnarray*}
The first and second components of the bigrading are called the {\em cohomological grading} and the {\em perverse grading}. The algebra structure is given by 
\begin{enumerate}
\item[(RC-1)] $\ep$ is central;
\item[(RC-2)] $\Sym(\frh),\Sym(\frh^{*})$ and $\QQ[W]$ are subalgebras;
\item[(RC-3)] For $\eta\in\frh, \xi\in\frh^{*}$ and $w\in W$, we have $w\cdot \eta=\leftexp{w}{\eta}\cdot w, w\cdot \xi=\leftexp{w}{\xi}\cdot w$ ;
\item[(RC-4)] For $\eta\in\frh,\xi\in\frh^{*}$, we have
\begin{equation*}
[\eta,\xi]=\jiao{\xi,\eta}\ep-\dfrac{\nu}{2}\left(\sum_{\alpha\in\Phi}\jiao{\xi,\alpha^\vee}\jiao{\alpha,\eta}r_{\alpha}\right)\ep,
\end{equation*}
where $\Phi$ is the set of roots in $\frh^{*}$, and $r_{\a}\in W$ is the reflection associate to a root $\a$. 
\end{enumerate}

The subalgebra $\Hrat_\nu^{sph}:=e\Hrat_\nu e$, $e=|W|^{-1}\sum_{w\in W} w$ is called spherical rational Cherednik algebra. By setting central element $\ep$ in $\Hrat_\nu$ to $1$ we obtain the algebra $\Hrat_{\nu,\epsilon=1}$ that is most commonly called the rational Cherednik algebra. Similarly we have the spherical version $\Hrat_{\nu,\ep=1}^{sph}$.

\subsubsection{}\label{sssection: perv}
 
The following theorem is a particular case of the main theorem in \cite{OY}.
\begin{theorem}\label{thm:L(triv)} Let $\nu=q/p>0$  and $(p,q)=1$. Then
\begin{enumerate}
\item There is a doubly-graded action of $\Hrat_{\nu}^{sph}$ on $\Gr^{P}_{*}\eqcoh{*}{\SG}$ under which $\ep$ acts as an equivariant parameter.
\item The $\Hrat_{\nu,\ep=1}^{sph}$-module $\Gr^{P}_{*}\spcoh{*}{\SG}$ is isomorphic to the irreducible spherical module $e\frL_{\nu}(\triv)$.
\end{enumerate}
\end{theorem}


\subsubsection{} Unlike $\Hrat_{q/p}$ the algebra $\HrtE$ has only one grading, the perverse grading. Moreover, the perverse grading is the internal grading by the action of the element $h:=\frac{1}{2}\sum_i(\xi_{i}\eta_{i}+\eta_i\xi_i)$ (where $\{\xi_i\}$ is a basis of $\frh^{*}$ and $\{\eta_{i}\}$ is the dual basis of $\frh$). Respectively we 
obtain the decomposition of $\frL_{q/p,\ep=1}(triv)$ into the eigenspaces of $h$: $\frL_{q/p,\ep=1}(\triv)=\oplus_{j=-\delta}^\delta \frL_{q/p,\ep=1}(\triv)[j]$ (where $\d=(p-1)(q-1)/2$).

The module \(\frL_{q/p,\ep=1}(\triv)\) is an irreducible quotient of the Verma module of \(\Sym(\frh^{*})\). Respectively, the spherical module \(e\frL_{q/p,\ep=1}(\triv)\) is the quotient
of \(\Sym(\frh^{*})^W=\CC[e_2,\dots,e_p]\) by an ideal \(I^{\frH}_{q/p}\). It was shown by Gorsky \cite[Theorem 4.3]{Gor} that \(I^{\frH}_{q/p}=(g_{q+1},\dots,g_{q+p-1})=I_{\mathcal{O}}\). We conclude that there is a canonical isomorphism that shifts the grading by $\d$
\begin{equation}\label{Oqp eL}
\cO_{q/p}\cong e\frL_{q/p,\ep=1}(\triv).
\end{equation}

\subsection{Relation with the Hilbert scheme}
\label{sec:sheaves-hilbn_0cc2}

Another model for the finite dimensional representation \(\frL_{q/p,\ep=1}(\triv)\)  stems from the interpretation of the rational Cherednik algebra as a quantization of the Calogero-Moser
space \cite{EG}. The Calogero-Moser space is a hyper-K\"ahler twist of the balanced Hilbert scheme  of points on the plane \(\Hilb^p_0(\AA^2)\), a scheme over $\QQ$ with $\CC$-points:
\[ \Hilb_0^p(\AA^2)(\CC)=\{\textup{ideals } I\subset \CC[x,y]|\dim \CC[x,y]/I=p, | supp(\CC[x,y]/I)|=(0,0)\},\]
where \(|supp(\CC[x,y]/I)|\) stands for the total sum of coordinates of the $n$ points in the support of $\CC[x,y]/I$. Let $Z_{p}\subset \Hilb^{p}_{0}(\AA^{2})$ be the subscheme whose $\CC$ points consists of ideals \(I\) such that $supp(\CC[x,y]/I)$ is concentrated at $(0,0)$.

There are two \(\Gm\) actions on \(\AA^2\). The first is anti-diagonal \(\lambda\cdot_{a}(x,y)=(\lambda x,\lambda^{-1} y)\) and the second dilates the second coordinate:
\(\lambda\cdot_{d}(x,y)=(x,\lambda y)\). We  use notations \(\Gm^a\) and \(\Gm^d\) for these tori.
The Picard group of \(\Hilb_0^p(\AA^2)\) has an ample generator denoted \(\mathcal{O}(1)\). We denote its restriction to $Z_{p}$ also by $\cO(1)$.

In the work of Gordon and Stafford \cite{GS}, the relation between the \(\Gm^a\times \Gm^d\)-equivariant coherent sheaves on \(\Hilb_0^p(\AA^2)\) and modules over \(\HrtE^{sph}\) was established.
In particular it is shown in \cite[Theorem 5.11]{GS} that there is an isomorphism of \(\ZZ\)-graded vector spaces for $k\ge0$
\begin{equation}\label{eq:H0Hilb}
e\frL_{1/p+k,\ep=1}(\triv)[j]\simeq \textup{H}^0(Z_{p}, \mathcal{O}(k))[j],\end{equation}
where the grading on the LHS corresponds to the action of \(\Gm^a\) on the RHS and \(j=-\delta,\cdots,\delta\).

Please, notice that statement from \cite{GS} is about the associate graded \(\Gr^\Lambda e\frL_{1/p+k,\ep=1}(\triv)\)
of the module \(e\frL_{1/p+k,\ep=1}(\triv)\) with respect to "the good filtration". In the presence of "the good filtration" \cite{GS} strengthen \ref{eq:H0Hilb} to the equality of
doubly graded spaces. However, \cite{GS} does not provide an explicit description of "the good filtration". The grading in (\ref{eq:H0Hilb}) is compatible with "the good filtration"
(see \cite[section 5.8]{GS}) thus the specialization (\ref{eq:H0Hilb}) of the statement from \cite{GS} is valid.

We propose to extend the above statement to interpret the graded dimension of the RHS in \eqref{eq:H0Hilb} with respect to the action of \(\Gm^d\), this extension
should be closely related to "the good filtration" from \cite{GS}:

\begin{conj}\label{conj:H0Hilb}
For $k\ge0$ and $i\ge0$ there is an isomorphism of vector spaces
\[\Gr_{\mathfrak{m}}^{i} e\frL_{1/p+k,\ep=1}(\triv)\simeq \textup{H}^0(Z_{p},\mathcal{O}(k))\langle i\rangle, \]
where the RHS is the weight $i$ part of $\textup{H}^0(Z_{p},\mathcal{O}(k))$ under the action of \(\Gm^d\). 
\end{conj}

\begin{exam}
In the case \(p=3,q=4\) (hence $k=1$), the structure of the finite dimensional representation \(e\frL_{1/p+k,\ep=1}(\triv)\)  is discussed in details in \cite[Section 5.2]{GORS}
and illustrated with \cite[Figure 1]{GORS}. In particular, in this we have
\begin{equation}\label{eq:cat5}
  \sum_i t^i\dim\mathfrak{m}^i/\mathfrak{m}^{i+1} =1+2t+t^2+t^3.
  \end{equation}

The space \(\textup{H}^0(Z_p,\mathcal{O}(k))\) was studied in papers of Garsia and Haiman. If \(k=1\) then the \(\Gm^d\)-character of the space specializes to
the Catalan number at \(t=1\) \cite[Section 2]{H1} and  these characters were tabulated  in \cite[Appendix]{GH}. When \(p=3\) the above mentioned \(t\)-Catalan number
coincides with the RHS of (\ref{eq:cat5}).
\end{exam}

\begin{prop}\label{p:Hilb imply Conj}
Conjecture~\ref{conj:H0Hilb} implies Conjecture~\ref{th:Gr m} for \(q=1+kp\).
\end{prop}
\begin{proof}
  It is shown by Gorsky and Mazin \cite[Theorem 3]{GM} that
   \[ \sum \dim \textup{H}^{2i}(J_{1/p+k}) t^i=\sum_{D}t^{|D|},\]
   where the sum is over the Young diagrams \(D\) that fit inside the triangle with vertices \((0,0)\), \((0,p)\), \((1+pk,0)\).
   
   The rational shuffle conjecture proved by Mellit \cite{Mel} combined with Haiman's results on the vanishing of higher cohomology \cite[Theorem 3.1]{H} and the
   localization formula for \(\mathcal{O}_Z\otimes \mathcal{O}(k)\) \cite[Proposition 3.2]{H} imply that 
   \[\sum_{D}t^{\delta-|D|}=\chi_t^{\Gm^d}(\textup{H}^0(Z_{p}, \mathcal{O}(k))).\]
Combining these facts and Conjecture \ref{conj:H0Hilb} we get
\begin{eqnarray*}
  \dim \cohog{2i}{J_{1/p+k}}&=&\dim\textup{H}^0(Z_{p},\mathcal{O}(k))\langle \delta-i\rangle \\
  &=&
  \dim \Gr^{\delta-i}_{\fm}e\frL_{1/p+k,\ep=1}(\triv)\stackrel{\eqref{Oqp eL}}{=}\dim \Gr^{\delta-i}_{\fm}\cO_{1/p+k}.
\end{eqnarray*}
Corollary \eqref{c:equiv conj} then implies that Conjecture~\ref{conj:Gr m} holds.

%
%
\end{proof}

\section{More conjectures}
\label{sec:conjectures}

We propose some conjectures that extend the results of the paper to more general curve singularities.

\subsection{Toric curves}
\label{sec:toric-curves}

A natural generalization of the curve \(C_{q/p}\) is a toric curve which we describe below. Let \(\Gamma\subset\ZZ_{\ge 0}\) be a sub-monoid, then we define \(\mathcal{O}_{\Gamma}\subset \CC[t]\)
to the ring with the basis \(t^i, i\in \Gamma\) and \(C^0_\Gamma=\Spec(\mathcal{O}_\Gamma)\).  We denote by \(C_\Gamma\) the projective curve which is the one-point compactification of \(C^0_\Gamma\) with a smooth point at infinity.

The curve \(C_{q/p}\) is isomorphic to the curve \(C_\Gamma\) with \(\Gamma=\langle p,q\rangle\). In general, the curve \(C_\Gamma\) is smooth outside of the origin $t=0$ and the singularity at the origin
has the tangent space of dimension equal to the minimal number of generators of \(\Gamma\). In particular, the singularity of \(C_\Gamma\) is planar iff \(\Gamma=\langle p,q\rangle\).

Similarly to the case of \(C_{q/p}\) one defines the one dimensional moduli space \(M_1((\mathbb{P}^1,\infty),(C_\Gamma,\infty))\) and its zero-dimensional rigidification
\(M^{rig}_1(\mathbb{P}^1,C_\Gamma)\). Our conjecture is concerned with the dimension of the Artinian ring \(\mathcal{O}_\Gamma:=\mathcal{O}(M^{rig}_1(\mathbb{P}^1,C_\Gamma))\):
\begin{conj}\label{conj:toric} We have
  \[\dim \mathcal{O}_\Gamma=\chi(J_\Gamma),\]
  where \(J_\Gamma\) is the moduli space of torsion free rank one coherent sheaves on $C_{\Gamma}$ with the same Euler characteristic as the structure sheaf $\cO_{C_{\Gamma}}$.
\end{conj}

The moduli space \(J_\Gamma\) has a natural \(\Gm\)-action by scaling $t$. By localizing to the $\Gm$-fixed points, 
the Euler characteristics on the RHS of the above conjecture is equal to the number of
\(\Gamma\)-submodules \(M\) of \(\ZZ\) which have the same relative size as \(\Gamma\). 
 The ring on the LHS of the conjecture
has an explicit description by generators and relations and one can compute its dimension with the help of software. In particular, we verified the above conjecture for many examples
of non-planar toric curves. Let us also point out that the conjecture is similar to \cite[Theorem 2]{FGS} where curves with planar singularities were studied.

\subsection{Planar curves}
\label{sec:toric-curves-1}

Another class of  curves that generalizes \(C_{q/p}\) is  the
class of complex curve with rational normalization and only one singularity that is planar. Given such a curve one can define the moduli
spaces  \(M_1((\mathbb{P}^1,\infty),(C,\infty))\),
\(M^{rig}_1(\mathbb{P}^1,C)\) (now over $\CC$); we discuss the construction of these spaces for particular cases of the curves below.

The topology of the compactified Jacobians of the curve from the above-mentioned class was studied by Piontkowski in \cite{Pion}. In particular in \cite{Pion} the Betti numbers of the compactified Jacobians of large series of non-toric curves were computed.
Let us also point out that it was conjectured by van Straten that the odd degree cohomology of the compactified Jacobians of such curves should vanish.

Using software we attempted to match the Betti numbers from \cite{Pion} with some algebraic invariants of the Artinian coordinate ring \(\mathcal{O}(M^{rig}_1(\mathbb{P}^1,C))\).
We discovered that the following weaker version of our Conjecture~\ref{th:Gr m} (or rather its equivalent form \eqref{H2i Grm}) survives numerical tests.

\begin{conj}\label{conj:planar} Let \(C\) be a curve with rational normailaztion and unique singularity that is planar.  Then for any \(i\in\ZZ\) we have
\begin{equation}\label{ineq Grm}
\dim \textup{H}^{2i}(\cJ(C))\ge \dim \Gr^{\d-i}_{\fm}\mathcal{O}(M^{rig}_1(\mathbb{P}^1,C))
\end{equation} 
where \(\mathfrak{m}\) is the maximal ideal in the Artinian ring $\mathcal{O}(M^{rig}_1(\mathbb{P}^1,C))$ and \(\delta=\dim \cJ(C)\).
\end{conj}

As we proved in Theorem \ref{th:main}(2), in the case \(C=C_{q/p}\),  that \eqref{ineq Grm} become an equality is equivalent to Conjecture~\ref{th:Gr m}. 

On the other hand in all non-toric case that we analyzed the inequality \eqref{ineq Grm} is strict for some $i$, in particular in these cases we see the inequality
\(\chi(\cJ(C))>\dim \mathcal{O}(M^{rig}_1(\mathbb{P}^1,C)) \). Let us give more details on the construction of the moduli space \(M^{rig}_1(\mathbb{P}^1,C)\) that we used in our
experiments.

%
%
%

\subsubsection{An example} Here we treat the simplest family of non-toric curves studied in \cite{Pion}, these curves are one point compactifications of the curves \(C^0_{4,2q,s}\) which have parametrization:
\[t\mapsto (t^4,t^{2q}+t^s),\]
where \(s>2q>4\)  and  \(q,s\) are odd.

The curve has a natural compactification \(C_{4,2q,s}\subset \mathbb{P}(4,s,1)\) with parametrization \(\phi_0(t_0:t_1)=(t_1^4:t_1^s+t_0^{s-2q}t_1^{2q}:t_0)\).
The image of \(\phi_0\) is the curve defined by the equation \(F_{q,s}(X_0,X_1,X_2)\) and it avoids the singularities of \(\mathbb{P}(4,s,1)\).
Hence the moduli space \(M^{rig}_1(\mathbb{P}^1,C_{4,2q,s})\) parametrizes maps:
\[\phi_{e,f}(t_0:t_1)=(t_1^4+\sum_{i=1}^4 e_i t_0^i t_1^{4-i}:t_1^s+\sum_{i=1}^s f_i t_0^i t_1^{s-i}:t_0),\]
such that \(f_{s-2q}=1\), \(4f_1=se_{1}\) (the rigidifying conditions) and \(F_{q,s}(\phi_{e,f}(t_0:t_1))=0\). 

Let us denote \(\dim \Gr_{\mathfrak{m}}^i\mathcal{O}(M^{rig}_1(\mathbb{P}^1,C_{4,2q,s}))\) as \(b^{fake}_{2(\delta-i)}\). We summarize the outcome of our computational experiments
in the tables below. The Betti numbers for the compactified Jacobians are taken from \cite{Pion}.

\medskip

\(q=3,s=7\):
\begin{center}
\begin{tabular}{l|l|l|l|l|l|l|l|l|l}
  \(  i\)         &0&1&2&3&4&5&6&7&8\\ \hline
  \(b_{2(\delta-i)}\)       &1&3&4&4&4&3&2&1&1\\
  \(b_{2(\delta-i)}^{fake}\) &1&3&4&4&4&2&1&0&0  
\end{tabular}
\end{center}

\(q=3,s=9\):
\begin{center}
\begin{tabular}{l|l|l|l|l|l|l|l|l|l|l}
  \(  i\)         &0&1&2&3&4&5&6&7&8&9\\ \hline
  \(b_{2(\delta-i)}\)       &1&3&4&4&4&4&3&2&1&1\\
  \(b_{2(\delta-i)}^{fake}\) &1&3&4&4&4&4&2&1&0&0  
\end{tabular}

\end{center}

\end{document}